\documentclass[a4paper,12pt]{article}
\usepackage{amsfonts,amsthm,amssymb,amsmath}
\usepackage[colorlinks=true,linkcolor=blue,citecolor=blue]{hyperref}

\usepackage{tcolorbox}
\usepackage[english]{babel}
\usepackage{amssymb,amsmath}
\usepackage{xcolor}

\newtheorem{Theo}{Theorem}[section]
\newtheorem{Prop}[Theo]{Proposition}
\newtheorem{Coro}[Theo]{Corollary}
\newtheorem{Lemm}[Theo]{Lemma}

\newtheorem{Exam}[Theo]{Example}

\newtheorem{Rema}[Theo]{Remark}

\newcommand{\ao}{\vec{A}}
\newcommand{\xo}{\vec{X}}
\newcommand{\yo}{\vec{Y}}

\def\T{\mathbb{ T}}
\def\N{\mathbb{ N}}
\def\R{\mathbb{ R}}

\begin{document}

\title{Subgaussian Kahane-Salem-Zygmund inequalities in Banach spaces}

\author{Andreas~Defant \,and\, Mieczys{\l}aw Masty{\l}o}

\date{}

\maketitle

\noindent
\renewcommand{\thefootnote}{\fnsymbol{footnote}}
\footnotetext{2010 \emph{Mathematics Subject Classification}: Primary 32A70, 60E15, 60G15, 60G50; Secondary  30K10, 46B70.}
\footnotetext{\emph{Key words and phrases}: Subgaussian random variables, random polynomial inequalities, Dirichlet polynomials,
interpolation functor}
\footnotetext{The second author was supported by the National Science Centre, Poland, Grant no. 2019/33/B/ST1/00165.}

\begin{abstract}
\noindent
The main aim of this work is to give a~general approach to the celebrated Kahane--Salem--Zygmund inequalities. We prove estimates for exponential Orlicz~norms of averages $\sup_{1\le j \leq N}  \big|\sum_{1 \leq i \leq K}\gamma_i(\cdot) a_{i,j}\big|\,,$ where  $(a_{i,j})$ denotes  a~matrix of scalars  and the  $(\gamma_i)$ a~sequence of  real or complex subgaussian random variables.~Lifting these inequalities to finite dimensional Banach spaces, we get novel Kahane--Salem--Zygmund type inequalities -- in particular, for  spaces of subgaussian random polynomials and multilinear forms
on finite dimensional Banach spaces as well as subgaussian random Dirichlet polynomials. Finally, we use abstract interpolation theory to widen our approach considerably.
\end{abstract}

\pagebreak

\tableofcontents

\section{Introduction}

\noindent
The study of random inequalities for trigonometric polynomials in one variable goes back to the seminal work of  Salem and Zygmund
in \cite{SalemZygmund}, and later it was Kahane who in \cite{Kahane} extended these ideas to more abstract settings -- including
trigonometric polynomials  in several variables. In the recent decades such inequalities have been  of central importance in numerous
topics of modern analysis, as, e.g., Fourier analysis, analytic number theory, or holomorphy in high dimensions.

In this work we attempt a~coherent abstract approach to subgaussian Kahane--Salem--Zygmund inequalities. Using  tools from
probability, Banach space, and interpolation theory, we improve  several probabilistic estimates which recently proved importance.
The suggested approach has main advantages -- it is powerful enough to derive novel results which in several relevant
cases turn out to be sharp.

Let us give a brief description of some keystones. Given a~sequence $(\gamma_i)_ {i \in\N}$ of subgaussian random variables over
a~probability measure space $(\Omega, \mathcal{A}, \mathbb{P})$ space, and a~finite sequence $(a_i)$ of vectors in $\ell_\infty^N$,
we are interested on estimates for the expectation of $\big\|\sum_i a_i \gamma_i\big\|_{\ell_\infty^N}$, which we call \emph{abstract
Kahane--Salem--Zygmund inequalities} ($K\!S\!Z$-inequalities for short).

\medskip

More generally and more precisely, we are looking for Banach function spaces $X$ over $(\Omega, \mathcal{A}, \mathbb{P})$ and Banach
sequence spaces $S = S(\N)$ such that, for every choice of finitely many vectors $a_1, \ldots, a_K \in \ell_\infty^N$ with $a_i =
(a_i(j))_{j=1}^N$, $1\leq i\leq K$, we have
\begin{equation}
\label{abstract}
\bigg\|\sup_{1\leq j\leq N}\Big|\sum_{i=1}^K a_i(j) \gamma_i\Big|\bigg\|_X
\leq \varphi(N) \sup_{1\leq j\leq N} \|(a_i(j))_{i=1}^K\|_S\,,
\end{equation}
where  $\varphi\colon \N \to (0, \infty)$ depends only on $X$ and $S$.

\medskip

In the following we want to explain why  we call such estimates abstract $K\!S\!Z$-inequalities. Note first that if we take a~sequence
$(\varepsilon_i)_{i \in \N}$ of independent  Rademacher random variables (that is, independent random variables  taking values $+1$ and
$-1$ with equal probability $\frac{1}{2}$), then for $X= L_r(\mathbb{P})$ with $1 \leq r < \infty$, $S = \ell_2$, and for $N=1$ the
estimate from \eqref{abstract}  reduces to Khintchine's inequality: There are constants $0< A_r \leq B_r <\infty$ such that, for each
$K\in \mathbb{N}$ and all scalars $t_1, \ldots, t_K$,
\begin{equation}
\label{kinchine}
A_r \, \|(t_i)\|_{\ell_2} \leq  \Big\|\sum_{i=1}^K t_i \varepsilon_i\Big\|_{L_r(\mathbb{P})} \leq
B_r\, \|(t_i)\|_{\ell_2}\,.
\end{equation}
Recall for $1 \leq r < \infty$ the definition of the exponential Orlicz function
\[
\varphi_r(t) = e^{t^r}-1, \quad\, t \in [0,\infty)\,,
\]
and for any random variable $f$ on $(\Omega, \mathcal{A}, \mathbb{P})$ the Orlicz norm
\begin{align*}
\|f\|_{L_{\varphi_r}} := \inf\bigg\{\varepsilon>0; \,\,\int_{\Omega}\,\varphi_r\Big(\frac{|f|}{\varepsilon}\Big)\,
d\mathbb{P} \leq 1\bigg\}.
\end{align*}
The Orlicz space $L_{\varphi_r}$ is the collection of all $f$ that satisfy $\|f\|_{L_{\varphi_r}} < \infty.$
We are going to use the following equivalent formulation of $L_{\varphi_r}$ in terms of $L_p$-spaces:
$f \in L_{\varphi_r}$ if only if  $f \in L_p$ for all $1 \leq p < \infty$ and $\sup_{1 \leq p < \infty} p^{-1/r}\|f\|_p< \infty $,
and in this case
\begin{equation}
\label{magic}
\|f\|_{L_{\varphi_r}} \asymp \sup_{1 \leq p < \infty} p^{-1/r}\|f\|_p\,,
\end{equation}
up to  equivalence with constants which only depend on $r$.

\medskip

In Theorem~\ref{matrix} we  prove, as one of our main results, that for  every $2 \leq r < \infty$ there is a~constant
$C_r >0$ such that, for each $K, N \in \mathbb{N}$ and for every choice of finitely many $a_1, \ldots, a_K \in \ell_\infty^N$,
with $a_i = (a_i(j))_{j=1}^N$, $1\leq i\leq K$, we have
\begin{equation}
\label{A}
\bigg\|  \sup_{1 \leq j \leq N} \Big|\sum_{i=1}^K \varepsilon_i  a_i(j)\Big|\bigg\|_{L_{\varphi_2}}
\leq C_2 (1 +\log N)^{\frac{1}{2}}\sup_{1 \leq j \leq N}\|(a_i(j)_{i=1}^K\|_{\ell_{2}}\,,
\end{equation}
and for $r \in (2, \infty)$
\begin{equation}
\label{B}
\bigg\|\sup_{1 \leq j \leq N} \Big|\sum_{i=1}^K  \varepsilon_i a_ i(j)\Big| \bigg\|_{L_{\varphi_r}}
\leq C_r (1 +\log N)^{\frac{1}{r}} \sup_{1 \leq j \leq N}\|(a_i(j)_{i=1}^K\|_{\ell_{r', \infty}}\,;
\end{equation}
here $\ell_{r', \infty}$ as usual indicates the classical Marcinkiewicz sequence space. Moreover, we will see that the
asymptotic behaviour of the constant $C_r (1 +\log N)^{\frac{1}{r}}$ can not be improved.

\medskip

Several remarks are in order. Note first that for $N=1$ and $r=2$ this estimate is due to Zygmund \cite{Zygmund}.
For $N=1$ and $r \in (2, \infty)$, Pisier in \cite{Pisier} proved that the Marcinkiewicz sequence space
$\ell_{r',\infty}$, instead of $\ell_2$, comes into play. Note that this fact was  mentioned by Rodin and Semyonov in
\cite[Section 6]{Rodin-Semyonov}. Observe that in view of \eqref{magic}, these estimates (still for $N=1$) obviously
extend the right-hand part of Kinchine's inequality.

Moreover, all estimates from \eqref{A} and \eqref{B} hold not only for Rademacher random variables, but even for the
much larger class of subgaussian random variables -- including real and complex normal Gaussian as well as complex
Steinhaus variables.

\medskip

Obviously both estimates have the form discussed in \eqref{abstract}, so let us come back to the above question why we decided
to call them 'abstract' $K\!S\!Z$-inequalities. Our main initial intention was to derive new multidimensional $K\!S\!Z$-inequalities.
The first estimates of this type were  studied by  Kahane who proves in \cite[pp.~68-69]{Kahane} that, given a~ trigonometric
Rademacher random polynomial $P$ in $n$ variables of degree  $\text{deg}(P) \leq m$, that is
\begin{align}
\label{ranpol}
P(\omega, z) = \sum_{|\alpha| \leq m} \varepsilon_\alpha(\omega) c_\alpha z^\alpha, \quad\,
\omega \in \Omega,\, z\in \mathbb{C}^n\,,
\end{align}
where  the $\varepsilon_\alpha$ for $\alpha \in \mathbb{Z}^n$ with $|\alpha| = \sum_k |\alpha_k| \leq m$ are independent
Rademacher variables on the  probability space $(\Omega, \mathcal{A}, \mathbb{P})$, the expectation of the sup norm of the
random polynomial on the $n$-dimensional torus $\mathbb{T}^n$ has the following upper estimate:
\begin{align}
\label{ksz}
\mathbb{E}\Big( \sup_{z \in \mathbb{T}^n} \big|P(\cdot, z) \big| \Big) \leq C\,\big(n (1+ \log m) \big)^{\frac{1}{2}}\,
\bigg(\sum_{|\alpha| \leq m} |c_\alpha|^2 \bigg)^{\frac{1}{2}}\,,
\end{align}
where $C > 0$ is a universal constant.

\medskip

Let us indicate how \eqref{A} implies  \eqref{ksz}. Denote by $\mathcal{T}_m(\T^n)$ the space of all trigonometric polynomials
$P(z) = \sum_{|\alpha| \leq m}  c_\alpha z^\alpha, \, z \in \mathbb{T}^n$ with $\text{deg}\,P \leq m$, which together with the
sup norm on $\mathbb{T}^n$ forms a~Banach space. A~well-known consequence of Bernstein's inequality (see, e.g., \cite[Corollary 5.2.3]{QQ})
is that, for all positive integers $n$, $m$ there is a~subset $F \subset \mathbb{T}^n$ of cardinality  $\text{card}\,F \leq (1+ 20 m)^n$
such that, for every $P \in \mathcal{T}_m(\T^n)$, we have
\begin{equation*} \label{bern-stein}
\sup_{z \in \mathbb{T}^n} |P(z)| \leq 2 \sup_{z \in F} |P(z)|\,.
\end{equation*}
In other terms, for  $N = (1 + 20 m)^n$ the linear mapping
\begin{align}
\label{bernd}
I\colon  \mathcal{T}_m(\T^n) \to \ell_\infty^N\,,\,\,\,I(P) := (P(z))_{z \in F}, \quad\, P \in \mathcal{T}_m(\T^n)
\end{align}
is an  isomorphic embedding satisfying $\|I\|\|I^{-1}\| \leq 2$\,. We now observe as an immediate consequence of \eqref{A}
and \eqref{B} that,  for each $2\leq r<\infty$, there exists a~constant $C_r >0$ such that, for any choice of polynomials
$P_1, \ldots, P_K \in \mathcal{T}_m(\mathbb{C}^n)$, we have
\begin{equation*}
\label{estimateINTRO4}
\bigg\|\sup_{z \in \mathbb{T}^n}\Big|\sum_{i=1}^K  \varepsilon_i P_i (z)\Big| \bigg\|_{L_{\varphi_2}}
\leq C_2 \big(n (1+ \log m) \big)^{\frac{1}{2}} \sup_{z \in \mathbb{T}^n}\big \| (P_i(z))_{i=1}^K\big\|_{\ell_2}\,,
\end{equation*}
and for $2 < r < \infty$
\begin{equation*} \label{estimateINTRO5}
\bigg\| \sup_{z \in \mathbb{T}^n}\Big|\sum_{i=1}^K  \varepsilon_i P_i(z) \Big|\bigg\|_{L_{\varphi_r}}
\leq  C_r\big(n (1+ \log m) \big)^{\frac{1}{r}} \sup_{z \in \mathbb{T}^n} \big\| (P_i(z))_{i=1}^K\big\|_{\ell_{r', \infty}}\,.
\end{equation*}
Applying this result to the Rademacher random polynomial  $P$ given by
\[
P(\omega, z) = \sum_{|\alpha| \leq m} \varepsilon_\alpha(\omega) c_\alpha z^\alpha
= \sum_{|\alpha| \leq m} \varepsilon_\alpha(\omega) P_\alpha(z), \quad\,\omega \in \Omega, \, z\in \mathbb{T}^n\,,
\]
we obviously get a strong extension of \eqref{ksz}, which can be seen as a sort of 'exponetial variant' of the
$K\!S\!Z$-inequality. Working out these ideas, we will show that this way various recent $K\!S\!Z$--inequalities for
polynomials and multilinear forms on finite dimensional Banach spaces can be simplified, unified, and extended --
in particular, recent results of Bayart \cite{Bayart} and Pellegrino et.~al.~\cite{PellegrinoSerranoSilva}.

\medskip

Using techniques from  the theory of  interpolation in Banach spaces, we further recover as well as extend our abstract
$K\!S\!Z$--inequalities like \eqref{A} and \eqref{B} considerably.

\medskip

Finally, in the last section we prove $K\!S\!Z$-inequalities for randomized Dirichlet polynomials. These results are heavily
based on 'Bohr's point of view', which shows an intimate interaction between the theory of Dirichlet polynomials and theory
of trigonometric polynomials in several variables.

\section{Preliminaries}

We use standard notation from Banach space theory. Let $X$, $Y$ be Banach spaces. We denote by $L(X, Y)$ the space of all
bounded linear operators $T\colon X\to Y$ with the usual operator norm. If we write $X\hookrightarrow Y$, then we assume that
$X\subset Y$ and the inclusion map ${\rm{id}}\colon X \to Y$ is bounded. If $X=Y$ with equality of norms, then we write $X\cong Y$.
We denote by $B_X$ the closed unit ball of $X$, and by $X^{*}$ its dual Banach space. Throughout the paper, $(\Omega, \mathcal{A},
\mathbb{P})$ stands for a~probability measure space. Given two sequences $(a_n)$ and $(b_n)$ of nonnegative real numbers we write
$a_n \prec b_n$ or $a_n=\mathcal{O}(b_n)$, if there is a constant $c>0$ such that $a_n \leq c\,b_n$ for all $n\in \mathbb{N}$,
while $a_n \asymp b_n$ means that $a_n \prec b_n$ and $b_n \prec a_n$ holds. Analogously we use the symbols $f\prec g$ and $f\asymp g$
for nonnegative real functions.\\

{\bf Function and sequence spaces.}
Let $(\Omega, \mu):= (\Omega, \Sigma, \mu)$ be a~complete $\sigma$-finite measure space and let $X$ be a~Banach space. $L^0(\mu, X)$
denotes the space of all equivalence classes of strongly measurable $X$-valued functions on $\Omega$, equipped with the topology of
convergence in measure (on sets of finite $\mu$-measure. In the case $X=\mathbb{K}$,  we write $L^0(\mu)$ for short instead of
$L^0(\mu,\mathbb{K})$ (where as usual $\mathbb{K}:=\mathbb{C}$ or $\mathbb{K}:= \mathbb{R}$). Let $E$ be a~Banach function lattice
over $(\Omega, \mu)$ and let $X$ be a Banach space. The K\"othe--Bochner space $E(X)$ is defined to consist of all $f\in L^0(\mu, X)$
with $\|f(\cdot)\|_X \in E$, and is equipped with the norm
\[
\|f\|_{E(X)}:= \|\,\|f(\cdot)\|_X\|_E\,.
\]
Recall that $E\subset L^0(\mu)$ is said to be a~Banach function lattice, if there exists $h\in E$ with $h>0$ a.e.
and $E$ is an Banach ideal in $L^0(\mu)$, that is, if $|f| \leq |g|$ a.e. with $g\in E$ and $f \in L^0(\mu)$, then
$f\in E$ and $\|f\|_E \leq \|g\|_E$.

By a Banach sequence space we mean a~Banach lattice in $\omega(\mathbb{N}):=L^0(\mathbb{N}, 2^{\mathbb{N}}, \mu)$,
where $\mu$ is the counting measure. A~Banach sequence space $E$ is said to be symmetric provided that
$\|(x_k)\|_E = \|(x^{*}_k)\|_E$, where $(x^{*}_k)$ denotes the decreasing rearrangement of the sequence $(|x_k|)$.
Given a Banach sequence space $E$ and a  positive integer $N$,
\[
\|(x_k)_{k=1}^N\|_{E^N}:= \Big\|\sum_{k=1}^N |x_k| e_k\Big\|_E, \quad\, (x_k)_{k=1}^N \in \mathbb{C}^N
\]
defines a norm on $\mathbb{C}^N$. In what follows we identify $(x_k)_{k=1}^N$ with $\sum_{k=1}^N x_k e_k$ and for
simplicity of notation, we write $\|(x_k)_{k=1}^N\|_E$ instead of $\|(x_k)_{k=1}^N\|_{E^N}$.

We will consider the Marcinkiewicz symmetric sequence spaces $m_w$. Recall that if $w= (w_k)$ is
a~non-increasing positive sequence, then $m_w$ is defined to be the space of all sequences $x= (x_k) \in
\omega(\mathbb{N})$ equipped with the norm
\[
\|x\|_{m_w} := \sup_{n\in \mathbb{N}} \frac{x_{1}^{*} + \cdots + x_{n}^{*}}{w_{1} + \cdots + w_{n}}\,.
\]
Note that if $\psi\colon [0, \infty) \to [0, \infty)$ is a concave function with $\psi(0)=0$, then $v:=(\psi(n)-\psi(n-1))$
is a~nonincreasing sequence sequence. It is easy to check that if $\lim\inf_{n\to \infty} \frac{\psi(2n)}{\psi(n)} >1$,
then there exists $C>1$
\[
\sup_{n\geq 1} \frac{n}{\psi(n)} x_{n}^{*} \leq \|x\|_{m_v} \leq C_r \sup_{n \geq 1} \frac{n}{\psi(n)} x_{n}^{*}\,.
\]
In particular, if $r\in (1, \infty)$ and $\psi(n)= n^{1-1/r}$, then the space $m_v$ coincides with the classical Marcinkiewicz
space $\ell_{r, \infty}$ and, in the above estimate, $C_r = r/(r-1)$.\\

{\bf Orlicz spaces.}
Let $\varphi \colon \mathbb{R}_{+} \to \mathbb{R}_{+}$ be an~Orlicz function
(that is, a~convex, increasing and continuous positive function with $\varphi(0)=0$). The Orlicz space $L_{\varphi}(\mu)$ ($L_{\varphi}$
for short) on a~measure space $(\Omega, \mu)$ is defined to be the space of all (real or complex) $f\in L^0(\mu)$ such that
$\int\,\varphi(\lambda |f|)\,d\mu <\infty$ for some $\lambda >0$, and it is equipped with the norm
\begin{align*}
\|f\|_{L_{\varphi}} = \inf\bigg \{\varepsilon>0; \,\,\int_{\Omega}\,\varphi\Big(\frac{|f|}{\varepsilon}\Big)\,
d\mu \leq 1\bigg\}\,,
\end{align*}
where in what follows, for simplicity of notation, we write $\int$ instead of $\int_{\Omega}$.

We will use the simple fact that whenever $(\Omega, \mathbb{P})$ is a probability measure space and two Orlicz functions
$\varphi$ and $\psi$ satisfy that $\varphi(t) \leq c \psi(t)$ for $t \ge t_0$, then $L_{\psi} \hookrightarrow L_{\varphi}$
with
\begin{equation}
\label{inclusion}
\|f\|_{\varphi} \leq (\varphi(t_0) + c)\,\|f\|_{\psi}, \quad\, f \in L_{\psi}\,.
\end{equation}
For $1 \leq r < \infty$, the  exponential Orlicz function
\[
\varphi_r(t) = e^{t^r}-1, \quad\, t \in [0,\infty)\,,
\]
is going to be  of particular interest. Clearly, for all $1 \leq r < \infty$
\[
L_{\varphi_r} \hookrightarrow L_r\,,\,\,\, \text{and \, \,$\|f\|_{L_{r}}  \leq \|f\|_{L_{\varphi_r}}$ for all $f \in L_{\varphi_r} $}\,.
\]
If $\Omega$ is a finite or countable set and $\mathcal{A}= 2^{\Omega}$, we write $\ell_{\varphi}(\mu)$ instead of $L_{\varphi}(\mu)$.\\

{\bf Polynomials.}
Given  Banach spaces $X_1, \ldots, X_m$, the product $X_1 \times \cdots \times X_m$ is equipped with the standard norm
$\|(x_1,\ldots, x_m)\| := \max_{1\leq j\leq m}\,\|x_j\|_{X_j}$, for all $(x_1,\ldots,x_m) \in X_1 \times \cdots \times X_m$.
The Banach space $\mathcal{L}_m(X_1, \ldots, X_m)$ of all scalar-valued $m$-linear bounded mappings $L$ on $X_1\times \cdots \times X_m$
is equipped with the norm
\[
\|L\|:= \sup\big\{|L(x_1,\ldots, x_m)\|;\, x_j \in B_{X_j}, \, 1\leq j\leq m\}\,.
\]
A scalar-valued function $P$ on a~Banach space $X$ is said  to be an $m$-homogeneous polynomial if it is the restriction of an $m$-linear
form $L$ on $X^m$ to its diagonal, i.e., $P(x) = L(x, \ldots, x)$ for all $x \in X$. We say that $P$ is a polynomial of degree at most
$m$ whenever $P = \sum_{k=0}^m P_k$, where all $P_k$ are $k$-homogeneous ($P_0$ a constant). For a~given positive integer $m$,  we denote
by $\mathcal{P}_{m}(X)$ the Banach space of all polynomials on $X$ of degree at most $m$ equipped with the norm
$\|P\| := \sup\{|P(z)|; \, z\in B_X\}$. The symbol $\mathcal{P}(X)$ denotes the union of all $\mathcal{P}_{m}(X), m \in \mathbb{N}$. More
generally, we write $\|P\|_E := \sup\{|P(z)|; \, z\in E\}$, whenever $E$ is a non-empty subset of $X$.

For a~multi-index $\alpha = (\alpha_1, \ldots\, \alpha_n) \in \mathbb{Z}^n$ and $z=(z_1, \ldots, z_n)\in \mathbb{C}^n$, the standard notation
$|\alpha|:= |\alpha_1| + \ldots + |\alpha_n|$ and  $z^{\alpha} := z_1^{\alpha_1} \cdots z_n^{\alpha_n}$ is used. For $\alpha
= (\alpha_1, \ldots\, \alpha_n) \in \mathbb{N}_{0}^n$, we let $\alpha! := \alpha_{1}! \cdots \alpha_{n}!$, where $\mathbb{N}_{0}:= \mathbb{N}
\cup \{0\}$. By $\mathbb{N}_{0}^{(\N)}$ we denote the union of all multi indices $\alpha \in \mathbb{N}_{0}^n, n \in \N$.

Given a finite dimensional Banach space $X = (\mathbb{C}^n, \|\cdot\|)$, every polynomial $P \in \mathcal{P}_m(X)$ has the form
$P(z)= \sum_{|\alpha| \leq m} c_\alpha z^\alpha,\, z\in \mathbb{C}^n$, and its degree is given by
$\text{deg}(P) := \max \{|\alpha|; \, c_\alpha \neq 0\}$.

For $n \in \mathbb{N}$ and $m \in \mathbb{N}_0$ we denote by $\mathcal{T}_m(\mathbb{T}^n)$ the space of all  trigonometric polynomials
$P(z) = \sum_{\alpha \in \mathbb{Z}^n, |\alpha| \leq m} c_\alpha z^\alpha$ on the $n$-dimensional torus $\mathbb{T}^n$ which have degree
$\text{deg}(P) = \max \{|\alpha|; \, c_\alpha \neq 0\} \leq m$. Clearly, $\mathcal{T}_m(\mathbb{T}^n)$ together with the sup norm
$\|\cdot\|_{\mathbb{T}^n}$ also denoted by $\|\cdot\|_{\infty}$) forms a Banach space.
\\

{\bf Interpolation.}
We recall  some fundamental notions from interpolation theory (see, e.g., \cite{BL, BK, Ov84}). The pair $\vec{X}=(X_0,X_1)$
of Banach spaces is called a~Banach couple if there exists a~Hausdorff topological vector space $\mathcal{X}$ such that
$X_j\hookrightarrow\mathcal{X}$, $j=0, 1$. A~mapping $\mathcal{F}$, acting on the class of all Banach couples, is called an
interpolation functor if for every couple $\xo = (X_0, X_1)$, $\mathcal{F}(\xo)$ is a~Banach space intermediate with respect
to $\xo$ (i.e., $X_0\cap X_1 \subset \mathcal{F}(\xo) \subset X_0 + X_1$), and $T\colon \mathcal{F}(\xo) \to \mathcal{F}(\yo)$
is bounded for every operator $T\colon \xo \to \yo$ (meaning $T\colon X_0 + X_1 \to Y_0 + Y_1$ is linear and its restrictions
$T\colon X_j \to Y_j$, $j=0,1$ are defined and bounded). If additionally there is a~constant $C>0$ such that for each
$T\colon \xo \to \yo$
\[
\|T\colon \mathcal{F}(\xo) \to \mathcal{F}(\yo)\| \leq C\,\|T\colon \xo \to \yo\|\,,
\]
where $\|T\colon \xo \to \yo\|:= \max\{\|T\colon X_0 \to Y_0\|, \, \|T\colon X_1 \to Y_1\|\}$, then $\mathcal{F}$ is called
bounded. Clearly, $C\geq 1$, and if $C=1$, then $\mathcal{F}$ is called exact.

Following \cite{Ov84}, the function $\psi_{\mathcal{F}}$ which corresponds to an exact interpolation functor $\mathcal{F}$ by
the equality
\[
\mathcal{F}(s\mathbb{R}, t\mathbb{R}) = \psi_{\mathcal{F}}(s,t) \mathbb{R}, \quad\, s, t>0
\]
is called the {characteristic function of the functor $\mathcal{F}$. Here $\alpha \mathbb{R}$ denotes $\mathbb{R}$
equipped with the norm $\|\cdot\|_{\alpha \mathbb{R}} := \alpha |\cdot|$ for $\alpha >0$.

For a bounded interpolation functor $\mathcal{F}$ we define the fundamental function $\phi_{\mathcal{F}}$ of $\mathcal{F}$ by
\[
\phi_{\mathcal{F}} (s, t) = \sup \|T\colon \mathcal{F}(\xo) \to \mathcal{F}(\yo)\|,  \quad\, s, t>0\,,
\]
where the supremum is taken over all Banach couples $\xo$, $\yo$ and all operators $T\colon \xo \to \yo$ such that
$\|T\colon X_0 \to Y_0\|\leq s$ and $\|T\colon X_1 \to Y_1\| \leq t$.

It is easy to see that $\phi_{\mathcal{F}}$ belongs to the set $\mathcal{Q}$ of all functions $\varphi\colon (0, \infty)
\times (0, \infty) \to (0, \infty)$, which are non-decreasing in each variable and positively homogeneous (that is,
$\varphi(\lambda s, \lambda t)= \lambda \varphi(s,t)$ for all $\lambda, s,t >0$).

\section{Gateway}

The following estimate for Rademacher averages in $\ell_\infty^N$ is considerable weaker than what we are going to prove
in Theorem \ref{matrix}, where we replace  Rademacher variables $\varepsilon_i$ by an sequences of subgaussian random variables
$\gamma_i$ and $L_r$-spaces by exponential Orlicz spaces $L_{\varphi_r}$. But its  proof is considerably simpler
than what is going to follow later -- though it still reflects some of the main  ideas of this article.

\begin{Theo}
\label{gateway}
Let $(\varepsilon_i)_{i \in \N}$ be a~sequence of independent Rademacher random variables. Then, for every $r \in [2,\infty)$, every
$N \in \mathbb{N}$, and every choice of finitely many $a_1, \ldots, a_K \in \ell_\infty^N$ with $a_i=(a_i(j))_{j=1}^N$,
$1\leq i\leq K$, we have
\[
\bigg(\mathbb{E} \Big\|\sum_{i=1}^K a_i \varepsilon_i\Big\|_{\ell_\infty^N} ^r\bigg)^{1/r}
\leq e^2 \sqrt{r}\,
(1+ \log N)^{\frac{1}{2}}
\, \sup_{1 \leq j \leq N}\|(a_{i}(j))_{i=1}^K\|_{\ell_{2}}\,.
\]
Moreover, if we denote by $C(N,r)$ the best constant in this inequality, then we have
\[
C(N,r) \asymp (1 + \log N)^{\frac{1}{2}}\,.
\]
\end{Theo}

Note that for $N=1$   these estimates (up to constants) are covered by (the right hand side)
of Khinchine's  inequality from  \eqref{kinchine}.  \\

For the proof we need slightly more  preparation. Define for each $N \in \mathbb{N}$  the $N$-th harmonic number
\[
h_N := \sum_{j=1}^N \frac{1}{j}\,,
\]
and the discrete probability measure $\mu_N$ on $\{1, \ldots, N\}$ by  $\mu_N(\{j\}) := \frac{1}{j}$. In what follows,
we will use the following obvious estimates without any further reference:
\begin{align*}
\label{log}
\log N < h_N \leq 1+\log N, \quad\, N \in \mathbb{N}\,.
\end{align*}
We add an elementary observation which will turn out to be crucial.}

\begin{Lemm} \label{kisliakov}
For every $\xi = (\xi_i)\in \mathbb{C}^N$, we have
\[
\frac{1}{e} \| \xi\|_{\ell_\infty^N} \leq \|\xi\|_{L_{h_N}(\mu_N)}
\leq  e^{\frac{1}{e}} \| \xi\|_{\ell_\infty^N}\,.
\]
\end{Lemm}

\begin{proof}
From the obvious inequality $\frac{\log t}{t} \leq \frac{1}{e},\,t\geq 1$, we get that
\[
\|\xi\|_{L_{h_N}(\mu_N)} \leq  \Big(\sum_{j=1}^N \frac{1}{j} \Big)^{\frac{1}{h_N}} \|\xi\|_{\ell_\infty^N}
= h_N^{\frac{1}{h_N}} \|\xi\|_{\ell_\infty^N} \leq e^{\frac{1}{e}} \|\xi\|_{\ell_\infty^N}\,.
\]
Conversely, if $\|\xi\|_{L_{h_N}(\mu_N)} = 1$, then $\frac{|\xi_j|^{h_N}}{j} \leq 1$ and so
\[
|\xi_j| \leq j^{\frac{1}{h_N}} \leq e^{\frac{1}{h_N} \log N} \leq e, \quad\, 1\leq j\leq N\,.
\]
This combined with the homogeneity of the norm yields the left hand estimate.
\end{proof}

We are ready for the proof of Theorem~\ref{gateway}.

\begin{proof}[Proof of Theorem~$\ref{gateway}$] By Lemma~\ref{kisliakov}
\begin{align*}
\bigg(\mathbb{E} \Big\|\sum_{i=1}^K a_i \varepsilon_i\Big\|_{\ell_\infty^N} ^r\bigg)^{1/r}
& \leq  e \bigg( \int\, \Big\|\Big(\sum_{i=1}^K  \varepsilon_i (\omega) a_i(j)
\Big)_{j=1}^N \Big\|^r_{L_{h_N}(\mu_N)}\, d\mathbb{P}(\omega)\bigg)^{\frac{1}{r}} \\
& = e\bigg(\int\, \Big(\sum_{j=1}^N  \Big|\sum_{i=1}^K  \varepsilon_i (\omega) a_i(j) \Big|^{h_N}
\frac{1}{j}\Big)^\frac{r}{h_N}\, d\mathbb{P}(\omega)\bigg)^{\frac{1}{r}} \\
& \leq  e\bigg( \int\, \Big(\sum_{j=1}^N  \Big|\sum_{i=1}^K  \varepsilon_i (\omega) a_i(j) \Big|^{h_N}
\frac{1}{j}\Big)^r\, d\mathbb{P}(\omega)\bigg)^{\frac{1}{r h_N}}\,,
\end{align*}
where the last estimate follows from H\"older's inequality.

Now the continuous Minkowski inequality implies
\begin{align*}
\bigg(\mathbb{E} \Big\|\sum_{i=1}^K a_i \varepsilon_i\Big\|_{\ell_\infty^N} ^r\bigg)^{1/r}
& \leq  e \bigg(\sum_{j=1}^N \Big(\int\, \Big|\sum_{i=1}^K  \varepsilon_i (\omega) a_i(j) \Big|^{r h_N} \frac{1}{j^r}\Big)^{\frac{1}{r}}\,d\mathbb{P}(\omega) \bigg)^{\frac{1}{ h_N}} \\
& =  e\bigg(  \sum_{j=1}^N  \frac{1}{j} \Big(\int\, \Big|\sum_{i=1}^K \varepsilon_i (\omega) a_i(j)
\Big|^{r h_N}\Big)^{\frac{1}{r}}\, d\mathbb{P}(\omega) \bigg)^{\frac{1}{ h_N}}\,.
\end{align*}
Finally, we use Kinchine's inequality \eqref{kinchine} together with the well-known estimate
$A_r \leq \sqrt{r}$ to get that
\begin{align*}
\bigg(\mathbb{E} \Big\|\sum_{i=1}^K a_i \varepsilon_i\Big\|_{\ell_\infty^N} ^r\bigg)^{1/r}
& \leq  e\bigg(  \sum_{j=1}^N  \frac{1}{j} \Big(\sqrt{r h_N}\, \big\|\big(a_i(j)\big)_{i=1}^K \big\|_2\Big)
^{h_N} \bigg)^{\frac{1}{ h_N}} \\
& \leq  e h_N^{\frac{1}{h_N}} \sqrt{r} (1+ \log N)^{\frac{1}{2}} \,\sup_{1 \leq j \leq N}\|(a_{i}(j))_{i=1}^K\|_{\ell_{2}}\,.
\end{align*}
Using the fact that $h_N^{\frac{1}{h_N}} \leq e$ gives the desired estimate. See the proof of the final argument in
Theorem~\ref{matrix} to check that the constant $C(2, N)$ is asymptotically optimal.
\end{proof}

From the norm equivalence  \eqref{magic} (take there $r = 2$) we immediately deduce the following consequence.

\begin{Coro} \label{2-case}
Let $(\varepsilon_i)_{i\in\N}$ be a~sequence of independent Rademacher random variables. Then, for any choice of finitely
many scalars $a_1, \ldots, a_K \in \ell_\infty^N$  with $a_i=(a_i(j))_{j=1}^N$, $1\leq i\leq K$, we have
\begin{equation*}
\label{strange}
\bigg\|  \sum_{i=1}^K  \varepsilon_i  a_i\bigg\|_{L_{\varphi_2}(\ell_\infty^N)}
\leq e^2 (1 + \log N)^{\frac{1}{2}}\sup_{1 \leq j \leq N}\big\|\big(a_i(j)\big)_{i=1}^K\big\|_{\ell_{2}}\,.
\end{equation*}
\end{Coro}

By a result of Peskir \cite{Peskir} it is known that for $N=1$ the  best possible constant here equals
$\sqrt{8/3}$.

\medskip

This means that $X=L_{\varphi_2}$ and $S = \ell_2$ in the language of \eqref{abstract} satisfy an abstract
$K\!S\!Z$--inequality with constant $\varphi(N)= e^2 (1 + \log N)^{\frac{1}{2}}$. In the following two sections
(Lemma~\ref{basic} and Theorem~\ref{matrix}) this result will be extended to $X=L_{\varphi_r}, 2 < r < \infty $,
$S = \ell_{r', \infty}$, and subgaussian random variables.
\\

\section{Subgaussian random variables}

Closely following Pisier \cite{Pisier2} we list some basic facts about real and complex subgaussian random variables,
and prove, in Lemma~\ref{basic}, one of our basic tools.

Let $(\Omega, \mathcal{A}, \mathbb{P})$ be a probability space, and $f$ a random variable. If $f$ is real-valued, then
$f$ is said to be subgaussian, whenever there is some $s \ge 0$ such that for every $x \in \mathbb{R}$
\[
\mathbb{E} \exp (xf) \leq \exp (s^2x^2/2)\,,
\]
and if $f$ is complex-valued, whenever there is some $s \ge 0$ such that for every $z \in \mathbb{C}$
\[
\mathbb{E} \exp ({\rm{Re}}(zf) \leq \exp (s^2|z|^2/2)\,.
\]
In this case, the best such $s$ is denoted by $\text{sg}(f)$. Note that subgaussian random variables always have mean zero.

By Markov's inequality it is well-known that, given a real subgaussian $f$, we for all $t > 0$ have
\begin{equation}\label{distr1}
\mathbb{P}\big(\{ |f| > t \}\big) \leq 2 \exp \bigg(\frac{- t^2}{ 2 \text{sg}(f)^2} \bigg)\,,
\end{equation}
whereas in the complex case
\begin{equation}\label{distr2}
\mathbb{P}\big(\{ |f| > t \}\big) \leq 4 \exp \bigg(\frac{- t^2}{ 4 \text{sg}(f)^2} \bigg)\,.
\end{equation}
Let us recall a few examples (\cite[Lemma 1.2 and p.5]{Pisier2}).

\begin{Exam} Of course, real and complex normal gaussian variables are subgaussian with constant $1$. Rademacher random variables
$\varepsilon_i$ are subgaussian with $\text{sg}(\varepsilon_i) =1$, and also the complex Steinhaus variables $z_i$ $($random variables
with values in the unit circle $\mathbb{T}$ and with distribution equal to the normalized Haar measure$)$ have this property with
$\text{sg}(z_i) = 1$.

Moreover, if $\gamma_1, \ldots, \gamma_n$ are  subgaussians $($real or complex$)$, then $\sum_{i=1}^n \gamma_i$ is subgaussian and
\begin{equation}\label{subgausformula}
\text{sg}\Big( \sum_{i=1}^n \gamma_i\Big) \leq \sqrt{2}  \Big( \sum_{i=1}^n \text{sg}(\gamma_i)^2 \Big)^{1/2}\,.
\end{equation}
\end{Exam}
The following lemma (see, e.g., \cite[Lemma 3.2]{Pisier2}) indicates that the Orlicz spaces $L_{\varphi_r}, 1 \leq r < \infty$ provide a~natural framework for the study
of subgaussian random variables.

\begin{Lemm}
A real mean-zero random variable $f$ is subgaussian if and only if $f \in L_{\varphi_2}$, in which case $\text{sg}(f)$ and
$\|f\|_{L_{\varphi_2}}$ are equivalent up to universal constants.
\end{Lemm}

\medskip

As discussed in the introduction the following result is one of our  crucial tools. In the case of Rademacher random variables 
see again Zygmund \cite{Zygmund} ($r=2$), Pisier \cite{Pisier}, and Rodin-Semyonov \cite{Rodin-Semyonov} (mentioned without proof). 
Replacing Rademacher random variables by sugaussians, it is an improvement of a~result mentioned by Pisier in \cite[Remark 10.5]{Pisier2}, 
and it is surely well-known to specialists. We include a~proof which is done in a~similar fashion as  in the case of Rademacher random 
varibales in \cite[Section 4.1]{LedouxTalagrand}.

\medskip

\begin{Lemm}
\label{basic}
Let $(\gamma_i)_{i\in \N}$ be a sequence of $($real or complex$)$ subgaussian random variables over $(\Omega, \mathcal{A}, \mathbb{P})$ such
that $s = \sup_i \text{sg}(\gamma_i) < \infty$.
\begin{itemize}
\item[{\rm(1)}] There is a~constant $C_2 = C(s) >0$ such that, for any choice of $($real or complex$)$ scalars $\alpha_1, \ldots, \alpha_n$
\[
\Big\| \sum_{i=1}^n \alpha_i \gamma_i \Big\|_{L_{\varphi_2}} \leq C_2 \,\|(\alpha_i)\|_2\,.
\]
\item[{\rm(2)}] Assume, additionally, that $M = \sup_i \|\gamma_i\|_\infty <~\infty$. Then for any $r\in (2, \infty)$ there is
a~constant $C_r = C(r,s, M) >0$ such that, for any choice of $($real or complex$)$ scalars $\alpha_1, \ldots, \alpha_n$
\[
\Big\|\sum_{i=1}^n \alpha_i \gamma_i \Big\|_{L_{\varphi_r}} \leq C_r \,\|(\alpha_i)\|_{r', \infty}\,.
\]
\end{itemize}
\end{Lemm}

Note again that by \cite{Peskir} in the case of Rademacher random variables $\varepsilon_i$ the  best constant $C_2$
is precisely $\sqrt{8/3}$.

\begin{proof}[Proof of Lemma~$\ref{basic}$]
We only discuss the real case -- the proof of the complex case is similar.

\medskip

\noindent ${\rm(1)}$
We  fix scalars $\alpha_1, \ldots, \alpha_n \in \mathbb{R}$ such that $\sum_{i=1}^n |\alpha_i|^2 =1$. From  \eqref{distr1}
and \eqref{subgausformula}, we deduce that for $f= \sum_{i=1}^n \alpha_i \gamma_i$,
\[
\mathbb{P}\big(\{|f| > t \}\big) \leq 2 \exp \Big(\frac{- t^2}{ 4 s^2}\Big)\,.
\]
Then, for every $c >0$, we have
\[
\mathbb{E}\Big(\varphi_2\big(|f|/c\big)\Big) = \int_0^\infty \mathbb{P}\big(\{|f| > ct \}\big) d(e^{t^2} -1)
\leq 4 \int_0^\infty t e^{t^2 - \frac{c^2t^2}{4s^2}}\,dt\,.
\]
Choosing $c = c(s)$ large enough, gives the conclusion.

\medskip

\noindent $(2)$
Take $r \in (2, \infty)$, and  $\alpha_1, \ldots, \alpha_n \in \mathbb{R}$ decreasing such $|\alpha_i| \leq |i|^{-1/r'}$
for each $1 \leq i \leq n$  (without loss of generality). We prove that for some constant $C_r = C(r,s,M) >0$ for all $t >0$
\begin{equation}
\label{point}
\mathbb{P}\big(\{|f| > t \}\big) \leq 2 \exp \Big(\frac{- t^{r}}{ C_r}\Big)\,,
\end{equation}
since then the conclusion follows as before.

We distinguish two cases, $t < 2Mr$ and $t \ge 2Mr$. In the first case, it is obvious (since $\mathbb{P}\big(\{|f| > t \}\big) \leq~1$)
that  there is such a constant $C_r >0$. In the second case, so $t \ge 2Mr$, we define $m(t) = \lfloor \big(\frac{t}{2Mr}\big)^r\rfloor$
and obtain
\[
|f| \leq M \sum_{i \leq m(t)} |\alpha_i| + \Big|\sum_{i> m(t)} \alpha_i \gamma_i \Big|
\leq M r m(t)^{1/r} + \Big |\sum_{i > m(t)} \alpha_i \gamma_i \Big|\,.
\]
(if $m(t) \ge n$, then the second sum is supposed to be $0$). Then, for all $t >0$, we get that
\[
\mathbb{P}\big(\{|f| > t \}\big) \leq \mathbb{P}\Big(\Big\{\Big|\sum_{i > m(t)} \alpha_i \gamma_i  \Big| > Mr m(t)^{1/r}\Big\}\Big)
\leq 2 e^{-\frac{M^2 r^2m(t)^{2/r}}{4 \sum_{i >m(t)} |\alpha_i|^2}}\,.
\]
Finally, using the fact that $|\alpha_i|^2 \leq i^{-2/r'}$ for all $i$, we see that there is some $C'_r = C'(r,s,M) >0$ such that
for all $t \ge 2Mr$,
\[
\frac{t^{r}}{ C'_r} \leq \frac{M^2 r^2m(t)^{2/r}}{4 \sum_{i >m(t)} |\alpha_i|^2}\,,
\]
and this competes the proof.
\end{proof}

\section{Subgaussian averages in $\ell_\infty^N$} \label{subgauss}

As promised above, we now extend the abstract $K\!S\!Z$--inequalities  from Theorem~\ref{gateway} and Corollary~\ref{2-case}.
In what follows we will need the following statement, which is  easily verified by using standard calculus.

\begin{Lemm}
\label{convex}
For any $c>0$ and $\alpha \in (0, 1)$ the function $\varphi$ given by $\varphi(t) := e^{(ct)^\alpha}-1$ for all
$t \in [0, \infty)$  is convex on the interval $\big[\big(\frac{1-\alpha}{\alpha}\big)^{\frac{1}{\alpha}}\frac{1}{c}, \infty\big)$.
In particular, for $c = \big(\frac{1}{\alpha}\big)^{\frac{1}{\alpha}}$ the function $\varphi$ is convex on $[1, \infty)$.
\end{Lemm}

The following theorem is the main result of this section. For $N=1$ it obviously recovers Lemma~\ref{basic},
being the crucial tool for the proof.

\begin{Theo}
\label{matrix}
Let $(\gamma_i)_{i\in \N}$ be a sequence of $($real or complex$)$ subgaussian random variables over
$(\Omega, \mathcal{A}, \mathbb{P})$ such that $s = \sup_i \text{sg}(\gamma_i) < \infty$.
\begin{itemize}
\item[{\rm(1)}] There is a~constant $C_2 = C(s) >0$ such that, for each $K, N \in \mathbb{N}$, and every choice of finitely
many $a_1, \ldots, a_K \in \ell_\infty^N$ with $a_i=(a_i(j))_{j=1}^N$, $1 \leq i \leq K$, we have
\begin{equation*}
\label{estimateII}
\bigg\|  \sum_{i=1}^K  \gamma_i  a_i\bigg\|_{L_{\varphi_2}(\ell_\infty^N)}
\leq C_2 (1 + \log N)^{\frac{1}{2}} \sup_{1 \leq j \leq N}\big\|\big(a_i(j)\big)_{i=1}^K\big\|_{\ell_2}\,.
\end{equation*}
\item[{\rm(2)}] Assume, additionally, that $M = \sup_i \|\gamma_i\|_\infty <~\infty$. Then for every $r \in (2, \infty)$ there
is a~constant $C_r=C(r,s,M) >0$ such that, for each $K, N \in \mathbb{N}$, and every choice of finitely many
$a_1, \ldots, a_K \in \ell_\infty^N$ with $a_i=(a_i(j))_{j=1}^N$, $1 \leq i \leq K$, we have
\begin{equation*}
\label{estimateII}
\bigg\|  \sum_{i=1}^K  \gamma_i  a_i\bigg\|_{L_{\varphi_r}(\ell_\infty^N)}
\leq C_r (1 + \log N)^{\frac{1}{r}} \sup_{1 \leq j \leq N}\big\|\big(a_i(j)\big)_{i=1}^K\big\|_{\ell_{r',\infty}}\,.
\end{equation*}
\end{itemize}
Moreover, for a fixed sequence $(\gamma_i)$ we denote  the best constant in ${\rm(1)}$ $($case $r=2)$ and ${\rm(2)}$
$($case $r \in (2, \infty))$  by $C(N,r)$. Then for normal Gausian,  Rademacher or  Steinhaus  variables we in ${\rm(1)}$ have
that $C(N,2) \asymp (1 + \log N)^{\frac{1}{2}}$, up to universal constants, whereas  in ${\rm(2)}$, we have that for Rademacher
or Steinhaus random variables $C(N,r) \asymp (1 + \log N)^{\frac{1}{r}}$, up to constants only depending on $r$.
\end{Theo}

\begin{proof}
We are going to handel the following two different cases separately, the first case: $\alpha(N)= \frac{r}{h_N} < 1,$ and the second:
$\alpha(N) = \frac{r}{h_N} \geq 1$. We start with the first case. By Lemma~\ref{kisliakov}, we have that
\begin{equation*}
\label{kisliakovII}
\bigg\|\sum_{i=1}^K  \gamma_i  a_i\bigg\|_{L_{\varphi_r}(\ell_\infty^N)}
\leq e \bigg\| \sum_{i=1}^K  \gamma_i  a_i\bigg\|_{L_{\varphi_r}(L_{h_N})}\,,
\end{equation*}
and so we estimate the second term. Fix $N\in \mathbb{N}$ and put
\[
c(N) := \Big(\frac{1}{\alpha(N)}\Big)^{\frac{1}{\alpha(N)}} \ge 1\,.
\]
Then we have
\begin{align*}
\int\, \varphi_r \Big( \Big(\sum_{j=1}^N \frac{1}{j} \Big|\sum_{i=1}^K &
\gamma_i  a_i(j) \Big|^{h_N} \Big) ^{\frac{1}{h_N}}\Big)\,d\mathbb{P} \\
& \leq \int\, \varphi_r \Big( \Big(c(N) \sum_{j=1}^N \frac{1}{jh_N} \Big|\sum_{i=1}^K
\gamma_i h_N^{\frac{1}{h_N}} a_i(j) \Big|^{h_N} \Big) ^{\frac{1}{h_N}}\Big) d \mathbb{P}\,.
\end{align*}
Define the function
\begin{equation*}
\psi_{r,N}(t) =
\begin{cases} \varphi_{\alpha(N)}(t) := \varphi_r(t^{\frac{1}{h_N}})\,, & \text{ if } t \geq 1\\[2ex]
\varphi_r(1)t\,, & \text{ if } t \leq 1.\\
\end{cases}
\end{equation*}
By Lemma \ref{convex} the function $t \mapsto \varphi_{\alpha(N)} (c(N)t)$ is convex on $[1, \infty)$,
and hence $t \mapsto \psi_{r,N} (c(N)t)$ is convex on  $[0, \infty)$. Then
\begin{align*}
\int\, \varphi_r \Big( \Big(\sum_{j=1}^N \frac{1}{j} \Big|\sum_{i=1}^K & \gamma_i
a_i(j) \Big|^{h_N} \Big) ^{\frac{1}{h_N}}\Big)\,d\mathbb{P} \\
& \leq \int\, \psi_{r,N} \Big( c(N)\sum_{j=1}^N \frac{1}{jh_N} \Big|\sum_{i=1}^K \gamma_i  h_N^{\frac{1}{h_N}} a_i(j)
\Big|^{h_N} \Big)\,d\mathbb{P} \\
& \leq \int\, \sum_{j=1}^N \frac{1}{jh_N} \psi_{r,N}   \Big( c(N) \Big|\sum_{i=1}^K   \gamma_i h_N^{\frac{1}{h_N}} a_i(j)
\Big|^{h_N} \Big) d \mathbb{P} \\
& \leq \int\, \sum_{j=1}^N \frac{1}{jh_N}\psi_{r,N}   \Big( \Big|\sum_{i=1}^K  \gamma_i c(N)^{\frac{1}{h_N}}h_N^{\frac{1}{h_N}}
a_i(j) \Big|^{h_N} \Big) d \mathbb{P}\,.
\end{align*}
Again changing the function, now with
\begin{equation*}
\tau_{r,N}(t) =
\begin{cases} \varphi_r(t) = \varphi_{\alpha(N)}(t^{h_N})\,, & \text{ if } t \geq 1\\[2ex]
\varphi_{r}(1)t\,, & \text{ if } t \leq 1\,,
\end{cases}
\end{equation*}
we obtain
\begin{align*}
\int\, \varphi_r \Big( \Big(\sum_{j=1}^N \frac{1}{j} & \Big|\sum_{i=1}^K \gamma_i  a_i(j) \Big|^{h_N} \Big) ^{\frac{1}{h_N}}\Big)
\,d\mathbb{P} \\
& \leq \int\, \sum_{j=1}^N \frac{1}{jh_N} \tau_{r,N}  \Big(\Big|\sum_{i=1}^K   \gamma_i  c(N)^{\frac{1}{h_N}}h_N^{\frac{1}{h_N}} a_i(j)
\Big| \Big)\,d\mathbb{P} \\
& \leq (\varphi_r(1)+1)\int \sum_{j=1}^N \frac{1}{jh_N} \varphi_r  \Big(\Big|\sum_{i=1}^K
\gamma_i c(N)^{\frac{1}{h_N}}h_N^{\frac{1}{h_N}} a_i(j) \Big| \Big)\,d\mathbb{P}\,,
\end{align*}
where we use \eqref{inclusion} in the last estimate. Finally, we arrive at
\begin{align*}
\label{ziel}
\int\,\varphi_r \Big(\Big(\sum_{j=1}^N \frac{1}{j}  \Big|\sum_{i=1}^K & \gamma_i  a_i(j) \Big|^{h_N} \Big)^{\frac{1}{h_N}}\Big)\,
d\mathbb{P} \\
& \leq e \sup_{1 \leq j \leq N} \int\, \varphi_r \Big(\Big|\sum_{i=1}^K \gamma_i c(N)^{\frac{1}{h_N}}h_N^{\frac{1}{h_N}} a_i(j)\Big|
\Big)\,d\mathbb{P}\,,
\end{align*}
which  immediately implies that
\begin{align*}
\bigg\|\sum_{i=1}^K  \gamma_i a_i(j)\bigg\|_{L_{\varphi_r}(\ell_\infty^N)} & \leq e \sup_{1 \leq j \leq N} \bigg\|\sum_{i=1}^K
\gamma_i  c(N)^{\frac{1}{h_N}}h_N^{\frac{1}{h_N}} a_i(j)  \bigg\|_{L_{\varphi_r}} \\
& \leq e h_N^{\frac{1}{r}}h_N^{\frac{1}{h_N}}  \sup_{1 \leq j \leq N}
\bigg\|\sum_{i=1}^K   \gamma_i  a_i(j)  \bigg\|_{L_{\varphi_r}}\,.
\end{align*}

Now using the exponential Khintchine inequality from Lemma \ref{basic}, we finish the proof of the first case $\alpha(N)= \frac{r}{h_N} < 1$.

Now we consider the second case: $\alpha(N)= \frac{r}{h_N} \geq  1$. Under this assumption on $N$, we show similarly as above  that the
inequality from  \eqref{estimateII} holds without any logarithmic term. Indeed, in this case we do not need any constant $c(N)$ (i.e., $c(N)=1$),
and replace the function $\psi_{r,N}(t)$ by $\varphi_{\alpha(N)}(t)$ itself (on all of $[0,\infty)$), which consequently is  automatically convex
on $[0, \infty)$. Now we go on, as above, with $\tau_{r,N}(t) = \varphi_{\alpha(N)}(t^{h_N}) = \varphi_r(t)$, and arrive finally at the above
estimates with $c(N) =1$, which finishes the argument.

Let us check the final result, and prove that
\[
C(N,r) \prec  (1 + \log N)^{\frac{1}{r}}\,,
\]
whenever we consider Rademacher variables $\varepsilon_i$. Indeed, we have that
\begin{align*}
\bigg(\mathbb{E} \Big\|\sum_{i=1}^K a_i \varepsilon_i \Big\|_{\ell_\infty^N}^2\bigg)^{1/2} & \leq \bigg\|  \sum_{i=1}^K  \varepsilon_i  a_i\bigg\|_{L_{\varphi_r}(\ell_\infty^N)} \\
& \leq C(N,r) \sup_{1 \leq j \leq N}\big\|\big((a_i(j)\big)_{i=1}^K\big\|_{\ell_{r', \infty}} \\
& \leq C(N,r) \sup_{1 \leq j \leq N} \Big(   \sum_{i=1}^K |a_i(j)|^{r'}   \Big) ^{\frac{1}{r'}}
\leq C(N,r) \Big(\sum_{i=1}^K \|a_i\|_{\ell_\infty^N}^{r'}   \Big) ^{\frac{1}{r'}}\,,
\end{align*}
which implies that
\[
T_{r'}(\ell_\infty^N) \prec  C(N,r)\,,
\]
where $T_{r'}(\ell_\infty^N)$ denotes the Rademacher type $r'$ of $\ell_\infty^N$ (which up to constants in $r$ equals
the Gaussian as well as the  Steinhaus type $r'$ of $\ell_\infty^N$). But it is well-known (see, e.g.,\,
\cite[p.16]{Tomczak-Jaegermann}) that up to universal constants, we have
\[
T_{r'}(\ell_\infty^N) \asymp (1 + \log N)^{\frac{1}{r}}\,,
\]
the conclusion. For all other cases the same proof works.
\end{proof}

\medskip

\section{Abstract $K\!S\!Z$--inequalities}

In this section we apply the abstract $K\!S\!Z$--inequality from Theorem~\ref{matrix} (see also again~\ref{abstract}) to trigonometric
polynomials, as well as polynomials and multilinear forms on Banach spaces.

The formulations, which distinguishes the apparently two different cases in this theorem, are somewhat cumbersome. This is the reason
why for simplicity of notation and presentation, we in the following remark make several agreements.

\begin{Rema} \label{notation}
All sequences $(\gamma_i)_{i\in I}$ of subgaussian random variables, with a~given countable set $I$ of indices, are defined over
a~probability measure space $(\Omega, \mathcal{A}, \mathbb{P})$. In each of our applications the varying index set $I$ will be clear
from the context.

The symbol $S_{r'}$ denotes the Hilbert space $\ell_2$, whenever $r =2$, and the Marcinkie\-wicz space $\ell_{r', \infty}$, whenever
$r\in (2, \infty)$. The space $S_{r'}$ is here understood as a Banach sequence space on a~corresponding countable set $I$ of indices.

If the sequence $(\gamma_i)_{i\in I}$ of subgaussians comes along with   $S_{r'},\,r\in [2, \infty)$,  we will always  assume that
$s = \sup_i \text{sg}(\gamma_i) < \infty$, and additionally $M = \sup_i \|\gamma_i\|_\infty < \infty$ whenever $r\in (2, \infty)$.
If in this case, the constant $C_r, r\in [2, \infty)$ appears,  then $C_2=C(s)$ will only depend on $s$, and
$C_r=C(r,s,M)$ only on $r, s, M$.
For  appropriate samples of all that we once again refer to Lemma~$\ref{basic}$ and Theorem~$\ref{matrix}$.
\end{Rema}

In what follows we need some definitions and facts from local  Banach space theory. Let $X$ and $Y$ be Banach spaces. An operator
$T\colon X\to Y$ is said to be an   isomorphic embedding of $X$ into $Y$ if there exists $C>0$ such that $\|Tx\|_Y \geq C \|x\|_X$
for every $x\in X$. In this case $T^{-1}$ is a well-defined operator from $(TX, \|\cdot\|_Y)$ onto $X$. Given a~real number
$1\leq \lambda <\infty$, we say that $X$, $\lambda$-embeds into $Y$ whenever there exists an isomorphic embedding $T$ of $X$ into
$Y$ such that
\[
\|T\|\,\|T^{-1}\| \leq \lambda\,.
\]
In this case, we call $T$ a $\lambda$-embedding of $X$ into $Y$. Observe that this is equivalent to the existence of a~set
$\{x_{1}^{*}, \ldots, x_{N}^{*}\}$ of functionals in $B_{X^{*}}$ such that for some $L, M>0$ with $L M \leq \lambda$, we have
\[
\frac{1}{L} \|x\|_X \leq \max_{1 \leq j\leq N} |x_j^{*}(x)| \leq M\|x\|_X, \quad\, x \in X\,.
\]
Then the  operator $T\colon X \to \ell_\infty^N$  given by
\[
Tx := (x_1^{*}(x), \ldots, x_N^{*}(x)), \quad\, x\in X
\]
induces the $\lambda$-embedding of $X$ into  $\ell_\infty^N$.

\medskip

The following remark will help to apply Theorem~\ref{matrix} in concrete cases.

\begin{Rema}
\label{sylvia}
Adopting the notation used in Remark~$\ref{notation}$, for every $r\in [2, \infty)$ there is a~constant $C_r >0$ such that, for
every  Banach space $E$, for every $\lambda$-embedding $I \colon E \hookrightarrow \ell_\infty^N$, and for every choice of
$x_1, \ldots, x_K \in E$, we have
\begin{equation*}
\label{estimateIIa*}
\bigg\|  \sum_{i=1}^K  \gamma_i x_i\bigg\|_{L_{\varphi_r}(E)} \leq \|I^{-1}\| \, C_r
(1 + \log N)^{\frac{1}{r}} \sup_{1 \leq j \leq N} \big\| (I(x_i)(j))_{i=1}^K\big\|_{S_{r'}}\,.
\end{equation*}
\end{Rema}

\medskip

Indeed, by Theorem~\ref{matrix} we have
\begin{align*}
\bigg\|\sum_{i=1}^K  \gamma_i x_i\bigg\|_{L_{\varphi_r}(E)} & \leq \|I^{-1}\|\bigg\| \Big(\sum_{i=1}^K \gamma_i I(x_i)(j)\Big)_{j=1}^N\bigg\|_{L_{\varphi_r}(\ell_\infty^N)} \\
& \leq C_r \|I^{-1}\| (1 + \log N)^{\frac{1}{r}}\sup_{1 \leq j \leq N}\big\| (I(x_i)(j))_{i=1}^K\big\|_{S_{r'}}\,.
\end{align*}

In view of this result, for a given finite dimensional Banach space $E$, we are interested in finding $\lambda$-embeddings
of $E$ into $\ell_{\infty}^N$ with the best possible dimension $N=N(\dim E, \lambda)$.

In this section we mainly concentrate on the Banach spaces $E= \mathcal{T}_m(\mathbb{T}^n)$, $E= \mathcal{P}_m(X)$, and
$\mathcal{L}_m(X_1, \ldots, X_m)$ (see again the preliminaries for the definitions). All coming estimates are based on the
following well-known result, which is a~consequence of a~volume argument (see \cite[Proposition 10, p.~74]{Woj} for details).

\begin{Prop}
\label{netball}
Let $E$ be an $n$-dimensional Banach space and $\varepsilon \in (0, 1)$. Then there exists an $\varepsilon$-net
$\{x_j\}_{j=1}^N$ in $B_E$ with $N\leq \big(1 + \frac{1}{\varepsilon})^{n}$ for real $E$, and
$N\leq (1 + \frac{1}{\varepsilon}\big)^{2n}$ for complex $E$.
\end{Prop}

The following corollary (see \cite[Proposition 13, p.76]{Woj}) is an immediate consequence.

\begin{Coro}
\label{varembedding}
For every $n$-dimensional Banach space $E$ and for every $\varepsilon \in (0, 1)$ there exists an isomorphic embedding
$I\colon E \to \ell_\infty^N$ with
\[
(1-\varepsilon) \|x\|_E \leq \|I(x)\|_{\ell_\infty^N} \leq \|x\|, \quad\, x\in E\,,
\]
where $N \leq \big(1 + \frac{1}{\varepsilon}\big)^n$ if $E$ is  a~real space, and $N\leq \big(1 + \frac{1}{\varepsilon}\big)^{2n}$
if $E$ is  a~complex space. In particular, we have that $I$ is an $(1-\varepsilon)^{-1}$-embedding.
\end{Coro}

For later use we collect another immediate consequence of Proposition \ref{netball}.

\begin{Coro}
\label{Kballs}
Let $E$ be an $n$-dimensional Banach space and $K \subset B_E$ a compact subset. Then for every $\varepsilon \in (0, 1)$
there exists a set $\{B(x_j, \varepsilon)\}_{j=1}^N$ of balls  with centers in $K$ covering $K$, where
$N \leq \big(1 + \frac{1}{\varepsilon}\big)^n$ in the real  and $N \leq \big(1+\frac{1}{\varepsilon})^{2n}$ in the complex case.
\end{Coro}

To see a first example at what we aim for, we mention the following abstract $K\!S\!Z$-inequality for $n$-dimensional Banach spaces
$E$, which is now an immediate consequence of Theorem~\ref{matrix} (in the form of Remark~\ref{sylvia}) and Corollary~\ref{varembedding}.

\begin{Theo}
\label{fits}
Adopting the notation used in Remark~$\ref{notation}$, for every $r \in [2, \infty)$ there is a~constant $C_r >0$ such that
for every $n$-dimensional Banach space $E$ and, for every choice of $x_1, \ldots, x_K \in E$, we have
\begin{equation*}
\label{estimateIIa}
\bigg\|  \sum_{i=1}^K \gamma_i x_i\bigg\|_{L_{\varphi_r}(E)} \leq C_r n^{\frac{1}{r}}
\sup_{\|x^\ast\|\leq 1} \big\|(x^\ast(x_i))_{i=1}^K)\big\|_{S_{r'}}\,.
\end{equation*}
\end{Theo}
We point out that it is easy to show that here the exponent in the term
$n^{\frac{1}{r}}$,  $2 \leq r < \infty$ can not be improved.

\subsection{Trigonometric polynomials}

Originally one of the initial motivations of this paper was to prove new general variants of Kahane--Salem--Zygmund random polynomial
inequalities, which recover the classical known results. We point out that in their seminal Acta paper Salem and Zygmund (see
\cite{SalemZygmund}; \cite[p.~69]{Kahane}) proved a~theorem for one-variable random trigonometric polynomials which states: Assume
that  $P_1, \ldots, P_K$ are trigonometric polynomials on $\T$ of degree at most $m$, and  $\gamma_1, \ldots, \gamma_K$ are independent
subgaussian random variables. Then, there exists a~universal constant $C>0$ such that
\[
\mathbb{P}\Big(\Big\{\omega \in \Omega\,;\,\, \Big\|\sum_{i=1}^N \gamma_i(\omega) P_i\Big\|_{\infty}
\geq C \Big(\sum_{i=1}^N \|P_i\|_{\infty} \log m\Big)^{\frac{1}{2}}\Big\}\Big) \leq \frac{1}{m^2}\,.
\]
There is a~large number of remarkable applications of this result, and in order to illustrate this, we comment two of them.

The first one, due to Odlyzko \cite{Odlyzko}, is related to the problem  of minimizing
\[
M(n) = \inf\Big\{-\min_{\theta \in [0, 2\pi]} \sum_{k=1}^\infty b_k \cos k\theta\Big\}\,,
\]
where the infimum is taken over all choices of $b_k\in \mathbb{N}_0$ with $\sum_{k=1}^{\infty} b_k=n$.

The Salem--Zygmund result was used in \cite{Odlyzko} to prove that given  any  trigonometric  cosine polynomial
$P(\theta) = b_0 + \sum_{k=1}^N b_k \cos k \theta,\,\theta \in [0, 2\pi]$, it is possible to change its coefficients slightly
so as to make  them  integers  without affecting the  values  of the  polynomial  too  severely. More precisely, let
\[
R(\omega,\theta) = \sum_{k=1}^N \xi_k(\omega) \cos k \theta, \quad\, \theta \in [0, 2\pi]\,,
\]
be the random cosine polynomial given by $\xi_k=0$ for each $1\leq k\leq N$, whenever $b_k$ is an integer, and else
$\mathbb{P}(\{\xi_k := \lfloor b_k\rfloor - b_k\}) = \lceil b_k \rceil - b_k$,
$\mathbb{P}(\{\xi_k := \lceil b_k \rceil - b_k\}) = b_k - \lfloor b_k \rfloor$. Then the Salem--Zygmund inequality yields that
\[
\lim_{N\to \infty} \mathbb{P}\big(\big\{\|R\|_\infty \prec (N \log N)^{\frac{1}{2}}\big\}\big) = 1\,,
\]
whereas the polynomial $P(\theta) + R(\omega, \theta)$ has always integer coefficients (except perhaps the constant coefficient).

Applying this random modification to the classical Fej\'er kernel, Odlyzko proved that
\[
M(n)= O\big((n \log n)^{\frac{1}{3}}\big)\,,
\]
and this leads, in particular, to improved  upper estimates for a problem of Erd\"os and Szekeres \cite{ErdosSzekeres} asking for
the largest possible value of all polynomials
\[
\prod_{k=1}^n (1-z^{\alpha_k}), \quad\, z \in \mathbb{T}
\]
on the unit circle $\mathbb{T}$ with $\alpha \in \N_0^n$.

The second application we wish to mention here, is related to the Hardy--Littlewood majorant problem for trigonometric polynomials
in $L_p(\mathbb{T})$ with $2<p \notin 2 \mathbb{N}$. In their remarkable paper \cite{MockenhauptSchlag}, Mockenhaupt and Schlag
proved a~version of the Salem--Zygmund inequality for asymmetric i.i.d. Bernoulli variables, and used it (in combination
with Bourgain's  results from \cite{Bourgain} on $\Lambda(p)$-Sidon sets), to show for each $N \in \N$ and $0<\rho <1$ the existence
of random sets $A \subset \{1,\ldots, N\}$ of size $N^{\rho}$ that satisfy, for all $\alpha >0$, the majorant inequality,
\[
\sup_{|a_n|\leq 1} \Big\|\sum_{n\in A} a_n z^n \Big\|_{L_p(\mathbb{T})} \leq C_\alpha N^{\alpha}
\Big\|\sum_{n\in A} z^n\Big\|_{L_p(\mathbb{T})}
\]
with a large probability.

\medskip

Let us come back to multidimensional Salem--Zygmund inequalities, first studied by Kahane (recall that we in short write
$K\!S\!Z$-inequalities). These  inequalities have numerous applications in many areas
of modern analysis as e.g. shown in \cite{Kahane}, and also \cite{Defant} and \cite{QQ}. Various variants were proved over
recent years, and what may be the most important one gives an upper bound of the expectation for the norm of random trigonometric
polynomials. As already indicated in the introduction, the following result is an extension of the $K\!S\!Z$-inequality for
random trigonometric Rademacher polynomials of degree less than or equal $m$ (see again \eqref{ranpol} and \eqref{ksz}).

\medskip

\begin{Theo}
\label{KSZone}
Adopting the notation used in Remark~$\ref{notation}$, for every $r\in [2, \infty)$ there is a constant $C_r >0$ such
that, for any choice of trigonometric  polynomials $P_1, \ldots, P_K \in \mathcal{T}_m(\mathbb{T}^n)$, we have
\begin{equation*}
\label{estimateIIab}
\bigg\| \sup_{z \in \mathbb{T}^n}\Big|\sum_{i=1}^K  {\gamma_i} P_i(z) \Big|\bigg\|_{L_{\varphi_r}}
\leq   C_r \big(n(1+\log m)\big)^{\frac{1}{r}} \sup_{z \in \mathbb{T}^n} \big\| (P_i(z))_{i=1}^K\big\|_{S_{r'}}\,.
\end{equation*}
\end{Theo}

\begin{proof}
This follows from the embedding in \eqref{bernd} and Theorem~\ref{matrix} (via a similar argument as in
Remark~\ref{sylvia}).
\end{proof}

The following corollary for subgaussian random polynomials is then obvious.

\begin{Coro} \label{KSZfour}
Adopting the notation used in Remark~$\ref{notation}$, for every $r\in [2, \infty)$ there is a~constant $C_r >0$ such that
for every random trigonometric polynomial $\sum_{|\alpha| \leq m} \varepsilon_\alpha c_\alpha z^\alpha \in \mathcal{T}_m(\mathbb{T}^n)$,
we have
\begin{equation*}
\bigg\| \sup_{z \in \mathbb{T}^n}\Big| \sum_{|\alpha|\leq m} \gamma_\alpha c_\alpha z^\alpha\Big| \bigg\|_{L_{\varphi_r}}
\leq C_r \big(n(1+\log m)\big)^{\frac{1}{r}} \big\| (c_\alpha )_{|\alpha|\leq m} \big\|_{S_{r'}}\,.
\end{equation*}
\end{Coro}

\subsection{Polynomials in Banach spaces}

In recent years many different types of extensions of the $K\!S\!Z$-inequality \eqref{ksz} were obtained, where the supremum
is taken over various Reinhard domains $R \subset \mathbb{C}^n$ (e.g., the unit ball $B_{\ell_p^n}$ of the Banach space
$\ell_p^n$, $1 \leq p < \infty$ instead of the $n$-dimensional torus $\mathbb{T}^n$).

Extending results from \cite{Boas} and \cite{DefantGarciaMaestre,DGMII}, Bayart in \cite{Bayart} estimates the expectation of
the norm of an $m$-homogeneous random Rademacher polynomial
\[
P(\omega, z) = \sum_{|\alpha| = m} \varepsilon_\alpha(\omega) c_\alpha z^\alpha
\]
on an arbitrary $n$-dimensional complex Banach space $X_n = (\mathbb{C}^n, \|\cdot\|)$. It is shown that, given
$r \in [2, \infty)$,
\begin{align*}
\mathbb{E}\Big( \sup_{z \in B_{X_n}} \big|P(\cdot, z) \big| \Big) \leq C_r \big(n(1+\log m)\big)^{\frac{1}{r}}
\sup_{|\alpha| =m} |c_\alpha| \Big( \frac{\alpha!}{m!}\Big)^{\frac{1}{r'}}
\sup_{z \in B_{X_n}} \Big( \sum_{i=1}^n |z_k|^{r'}\Big)^{\frac{m}{r'} }\,,
\end{align*}
where $C_r >0$ is a~constant only depending on $r$.

To prove results of this type, Bayart uses  two different methods. The first method is based on Khintchine-type inequalities
for Rademacher processes, and the second relies on controlling increments of a~Rademacher process in an Orlicz space, and in
this case an entropy argument is used.

We mention that Bayart applied his results in the study of multidimensional Bohr radii, as well as unconditionality in Banach
spaces  of homogenous polynomials; all this is also collected in the recent monograph \cite{Defant}. Finally,
we recall that \cite{DefantMastylo} and \cite{MastyloSzwedek} have several extensions of Bayart's results -- two articles
depending heavily on abstract interpolation theory.

In the following theorem, based on the abstract $K\!S\!Z$-inequality from Theorem~\ref{matrix}, we extend several of these
results -- in particular those obtained by Bayart's first method.

\begin{Theo}
\label{KSZtwo}
Adopting the notation used in Remark~$\ref{notation}$, for every $r \in [2, \infty)$ there is a~constant $C_r >0$ such that
for every $m \in \mathbb{N}_0, n \in \mathbb{N}$, every complex  $n$-dimensional  Banach space $X$, and every choice of
polynomials
$P_1, \ldots, P_K \in \mathcal{P}_m(X)$, we have
\begin{equation*}
\label{estimateIIab}
\bigg\| \sup_{z \in B_{X}}\Big|\sum_{i=1}^K  \gamma_i P_i(z) \Big|\bigg\|_{L_{\varphi_r}}
\leq  C_r \big(n(1+\log m)\big)^{\frac{1}{r}} \sup_{z \in B_{X}} \big\| (P_i(z))_{i=1}^K\big\|_{S_{r'}}\,.
\end{equation*}
\end{Theo}

We start the proof with another definition. Given a real or complex Banach space $X$ and a~compact set $K \subset B_X$,
we say that $K$ satisfies a~Markov--Fr\'echet inequality whenever there is an exponent $\nu \ge 0$, and a~constant $M >0$
such that, for every $P \in \mathcal{P}(X)$, we have
\[
\sup_{z \in K} \|\nabla P (z)\|_{X^\ast} \leq M (\text{deg}P)^{\nu} \sup_{z \in K} |P(z)|\,,
\]
where $\nabla P (z) \in X^\ast$ denotes the Fr\'echet derivative of $P$ in $z \in K$. If this inequality only holds for
a~subclass $\mathcal{P}$ of $\mathcal{P}(X)$, then we say that $K$ satisfies a~Markov-Fr\'echet inequality for $\mathcal{P}$
with exponent $\nu$ and  constant $M$.

\begin{Lemm}
\label{bernstein2}
Let $X$ be an $n$-dimensional Banach space $($real or complex$)$, and $K \subset B_X$ a~convex and compact set, which satisfies
a~Markov--Fr\'echet inequality with exponent $\nu$ and constant $M$. For each $m \in \N$ there exists a~subset $F \subset K$
such that, for every $P \in \mathcal{P}_m(X)$, we have
\[
\|P\|_K \leq 2 \sup_{z \in F}|P(z)\|_F\,,
\]
with $\text{card}\,F \leq N$, where $N = \big(1+2 M m^{\nu}\big)^{n}$ in the real case and $N  = \big(1+2Mm^{\nu}\big)^{2n}$ in
the complex case. In other words the Banach space  $\mathcal{P}_m(X)$, $2$-embeds into $\ell_\infty^N$.
\end{Lemm}

\begin{proof}
We assume that $X$ is complex, and take $P \in \mathcal{P}_m(X)$ (the real case follows the same way). Then for $z_{1},
z_2 \in K$ we obtain, using the fact that $K$ is convex and satisfies a~Markov--Fr\'echet inequality,
\begin{align*}
|P(z_1) - P(z_2)| & = \bigg|\int_0^1 \frac{d}{dt} P_k\big(tz_1+ (1-t) z_2\big)\,dt \bigg| \\
& =  \bigg|\int_0^1 \Big((\nabla P)\big(tz_1+ (1-t) z_2 \big)\Big)(z_1-z_2)\,dt\bigg| \\
& \leq M  m^{\nu}  \|P\|_{[z_1,z_2]} \, \|z_1-z_2\|_X \leq  M m^{\nu} \|P\|_{K} \, \|z_1-z_2\|_X\,.
\end{align*}
Applying Corollary~\ref{Kballs} with $\varepsilon := \frac{1}{2M m^{\nu}}$, we conclude that there is a~finite set $F\subset K$ with
$\text{card}\,F \leq (1+ 2M m^{\nu})^{2n}$ such that
\[
K \subset \bigcup_{u \in F} B_X(u, \varepsilon)\,.
\]
Then, for every $z \in K$ there is $v \in F$ with $\|z-v\| \leq \varepsilon$, which yields
\[
|P(z)| \leq |P(v)| + |P(z) - P(v)| \leq \sup_{u \in F} |P(u)|  + \frac{1}{2} \|P\|_{K}\,,
\]
and the proof is complete.
\end{proof}

When does the unit ball $B_X$ of a complex Banach space $X$ itself satisfy a~Markov--Fr\'echet inequality? For later use,
we collect a~few results in this direction, and start with the following result due to Harris \cite[Corollary 3]{Harris}.

\begin{Lemm}
\label{harry1}
Let $X$ be a complex Banach space. Then $B_X$  satisfies a~Markov--Fr\'echet inequality with constant $M=e$ and exponent
$\nu=1$.
\end{Lemm}

For our purposes it will be enough to know that this result holds with exponent $\nu = 2$, and for this weaker fact we include
a~self-contained proof.

\begin{proof}[Proof of Lemma $\ref{harry1}$ with exponent $\nu=2$] Take $P \in \mathcal{P}(X)$ with $m =\text{deg}\,P$, and consider
its Taylor  expansion $P = \sum _{k=0}^m P_k$ with $P_k \in  \mathcal{P}_k(X)$ (see, e.g.,  \cite[15.4]{Defant}). For each
$1 \leq k \leq m$ denote by $\check{P}_k$ the unique symmetric $m$-linear form on $X$ associated to $P_k$. Then by polarization
(see, e.g., \cite[(15.18)]{Defant}), for each $2 \leq k \leq n$ and for all $z, h \in B_X$, we have
\[
|(\nabla P_k(z))(h)| = k |\check{P}_k\big(z, \ldots, z, h  \big)| \leq k \Big(\frac{k}{k-1} \Big)^{k-1} \|P_k\|_{B_X}\,,
\]
and whence
\[
\sup_{z \in B_X} \|\nabla P_k(z)\|_{X^\ast} \leq e k \|P_k\|_{B_X}\,.
\]
This combined with the Cauchy inequality (see, e.g., \cite[Proposition 15.33]{Defant}) yields
\begin{align*}
\sup_{z \in B_X} \|\nabla P(z)\|_{X^\ast} & \leq \sum_{k=1}^m \sup_{z \in B_X} \|\nabla P_k(z)\|_{X^\ast} \\
& \leq \sum_{k=1}^m  k e  \|P_k\|_{B_X} \leq e m^2 \|P\|_{B_X}\,,
\end{align*}
and so the required estimate follows.
\end{proof}

Finally, we are ready to give the

\begin{proof}[Proof of Theorem~$\ref{KSZtwo}$]
Consider the $2$-embedding of the space $E = \mathcal{P}_m(X)$ into $\ell_\infty^N$ proved in  Lemma~\ref{bernstein2}.
Then Theorem \ref{KSZtwo} is an immediate consequence of Theorem~\ref{matrix} (in the form of Remark~\ref{sylvia}) observing
that every $z \in B_{X}$ defines a~norm one functional $x^{\ast} \in  E^\ast$ by $x^\ast(P) = P(z)$.
\end{proof}

Lemma \ref{harry1} is a result on complex Banach spaces $X$.  For real $X$ the proof of Lemma \ref{harry1} does not work, since
then no Cauchy inequality with constant $1$ is available, that is, the projection which assigns to each polynomial its $k$-th
Taylor polynomial is not contractive on $\mathcal{P}_m(X)$.

\medskip

However, applying the idea of the preceding proof to homogeneous polynomials only and using the 'real polarization estimate'
of Harris from \cite[Corollary 7]{Harris}, we get, in the homogeneous case, the following real variant of Lemma \ref{harry1}.

\begin{Lemm}
\label{harry2}
Let $X$ be a real Banach space. Then $B_X$  satisfies a~Markov--Fr\'echet inequality  for all homogeneous polynomials with
constant $M=\sqrt{e}$ and exponent $\nu=1/2$.
\end{Lemm}

\begin{Rema} \label{real}
Lemma~$\ref{harry2}$ combined with  Lemma~$\ref{bernstein2}$ shows that Theorem~$\ref{KSZtwo}$ holds for real Banach spaces
$X$ and  a~real sequence $(\gamma_i)$ of subgaussian random variables if we replace the space $\mathcal{P}_m(X)$ by its
subspace of all $m$-homogeneous polynomials.
\end{Rema}

Another  real result may be of interest. To state it, we recall that for every convex and compact set $C \subset \mathbb{R}^n$
the minimal width of $C$ is given by
\[
w(C)= \min \{w(u);\, \|u\|_2=1\}\,,
\]
where $w(u)$ is the width of $C$ in the direction of the normal vector $u \in \mathbb{R}^n,\,\|u\|_2 =1$.

\begin{Theo}
Adopting the notation used in Remark~$\ref{notation}$ together with the additional assumption that all subgaussians  $\gamma_i$
are real, for every $r\in [2, \infty)$ there is a~constant $C_r >0$ such that, for every convex and compact subset $C$ in $\R^n$
with non-empty interior, and every choice of polynomials $P_1, \ldots, P_K$ of degree $\leq m$ on  $\mathbb{R}^n$, we have
\begin{equation*}
\label{estimateIIab}
\bigg\| \sup_{x \in C}\Big|\sum_{i=1}^K  \gamma_i P_i(x) \Big|\bigg\|_{L_{\varphi_r}} \leq  C_r \bigg(n\Big(1+\log \Big(\frac{8m^2}{w(C)}\Big)\Big)\bigg)^{\frac{1}{r}} \sup_{x \in B_{\ell^n_2}} \big\| (P_i(x))_{i=1}^K\big\|_{S_{r'}}\,.
\end{equation*}
\end{Theo}

\begin{proof}
This result is a consequence of Theorem~\ref{matrix} in combination with Lemma \ref{bernstein2}, since a~remarkable result due to
Wilhelmsen \cite[Theorem 3.1]{Wilhelmsen} states that a~convex and compact subset $C$ of the real Hilbert space  $\ell_2^n$ with
non-empty interior satisfies a~Markov--Fr\'echet inequality with constant $M = 4/w(C)$ and exponent $\nu=2$.
\end{proof}

Fixing a basis in $X$, that is looking at $X =(\mathbb{C}^n, \|\cdot\|)$, we finally (like in Corollary~\ref{KSZfour}) list
another two corollaries of Theorem~\ref{KSZtwo} for subgaussian random polynomials.

\begin{Coro}
\label{KSZfive}
Adopting the notation used in Remark~$\ref{notation}$, for every $r\in [2, \infty)$  there is a~constant $C_r >0$ such that
for every Banach space  $X =(\mathbb{C}^n, \|\cdot\|)$, and every random polynomial
$\sum_{|\alpha| \leq m} \gamma_\alpha c_\alpha z^\alpha \in \mathcal{P}_m(X)$, we have
\begin{equation*}
\bigg\|  \sup_{z \in B_{X}}\Big| \sum_{|\alpha|\leq  m} \gamma_\alpha c_\alpha z^\alpha\Big| \bigg\|_{L_{\varphi_r}}
\leq C_r \big(n(1+\log m)\big)^{\frac{1}{r}} \sup_{z \in B_{X}}
\big\| \big(c_\alpha z^\alpha\big)_{|\alpha|\leq m} \big\|_{S_{r'}}\,.
\end{equation*}
\end{Coro}

We remark that for $X_n = \ell_p^n$ with $1 \leq p \leq \infty$, we have (see, e.g., \cite[Lemma~1.38]{Dineen})
\[
\sup_{z \in B_{\ell_p^n}} \big\|\big(c_\alpha z^\alpha\big)_{|\alpha|\leq m} \big\|_{S_{r'}}
\leq \sup_{|\alpha|\leq m} \bigg(\frac{\alpha^\alpha}{|\alpha|^{|\alpha|}}\bigg)^{\frac{1}{p}}\,
\big\|\big(c_\alpha \big)_{|\alpha|\leq m} \big\|_{S_{r'}}\,.
\]

\medskip

We finish showing that Corollary \ref{KSZfive} gives a~considerable extension of Bayart's result from
\cite[Theorem 3.1]{Bayart}.

\begin{Coro}
\label{KSZsix}
Adopting the notation used in Remark~$\ref{notation}$, for every $r \in [2, \infty)$ there is a~constant
$C_r >0$ such that, for every Banach space $X =(\mathbb{C}^n, \|\cdot\|)$ and every random polynomial
$\sum_{|\alpha| \leq m} \gamma_\alpha c_\alpha z^\alpha \in \mathcal{P}_m(X)$, we have
\begin{align*}
\begin{split}
\bigg\|  \sup_{z \in B_{X}}\Big| & \sum_{|\alpha| \leq m} \gamma_\alpha c_\alpha z^\alpha\Big| \bigg\|_{L_{\varphi_r}} \\
& \leq C_{r} \big(n(1+\log m)\big)^{\frac{1}{r}} \bigg( \sum_{k=0}^m \sup_{|\alpha| =k} \frac{ |c_\alpha|^{r'} \alpha!}{k!}
\sup_{z \in B_{X}} \Big( \sum_{i=1}^n |z_i|^{r'}\Big)^k\bigg)^{\frac{1}{r'}}\,.
\end{split}
\end{align*}
\end{Coro}

\begin{proof}
In view of Corollary~\ref{KSZfive}, all we have to show is that for $2 \leq r < \infty$ and $z \in \mathbb{C}^n$
\[
\Big(\sum_{|\alpha|\leq m} |c_\alpha z^\alpha|^{r'}\Big)^{\frac{1}{r'}}
\leq \bigg( \sum_{k=0}^m \sup_{|\alpha| =k} \frac{ |c_\alpha|^{r'} \alpha!}{k!}
\sup_{z \in B_{X}} \Big( \sum_{i=1}^n |z_i|^{r'}\Big)^k\bigg)^{\frac{1}{r'} } \,,
\]
and hence we check that for each $k\in \{1, \ldots, m\}$
\[
\sum_{|\alpha|=k} |c_\alpha z^\alpha|^{r'}
\leq \sup_{|\alpha| =k} \frac{ |c_\alpha|^{r'} \alpha!}{k!}
\sup_{z \in B_{X}} \Big( \sum_{i=1}^n |z_i|^{r'}\Big)^k \,.
\]
To understand this we need a~bit more of  notation. Following \cite{Bayart} or \cite{Defant}, for each positive integers
$m$ and $n$, we define
\begin{align*}
& \mathcal{M}(k,n) :=\{1, \ldots, n\}^{k}\,, \\
& \mathcal{J}(k,n) :=\{{{\bf{j}}} = (j_1, \ldots, j_k) \in \mathcal{M}(k, n); \, j_1 \leq \ldots \leq j_k\}\,.
\end{align*}
We consider on $\mathcal{M}(k,n)$  the equivalence relation: ${\bf{i}}  \sim {\bf{j}}$ if there is a permutation
$\sigma$ on $\{1, \ldots,k\}$ such that $(i_1, \ldots, i_k) = (i_{\sigma(1)}, \ldots, i_{\sigma(k)})$. The equivalence
class of ${\bf{i}} \in \mathcal{M}(k,n)$ is denoted by $[{\bf{i}}]$, and its cardinality  by $|[\bf{i}]|$. Obviously there
is a~canonical bijection between $\mathcal{J}(k,n)$ and the set of all multi-indices $\alpha \in \mathbb{N}_0^n$ with
$|\alpha|=k$, and if ${\bf{j}}$ is associated with $\alpha$, then $|[{\bf{j}}]| = k!/\alpha!\,.$ Then
\begin{align*}
\sum_{{\bf{j}} \in \mathcal{J}(k,n)} |c_{{\bf{j}}} z_{{\bf{j}}}|^{r'} & \leq \sup_{{\bf{j}} \in \mathcal{J}(k,n)}
\frac{|c_{{\bf{j}}}|^{r'}}{|[{{\bf{j}}}]|} \sum_{{\bf{j}} \in \mathcal{J}(k,n)}|[{{\bf{j}}}]| | z_{{\bf{j}}}|^{r'} \\
& = \sup_{{\bf{j}} \in \mathcal{J}(k,n)} \frac{|c_{{\bf{j}}}|^{r'}}{|[{{\bf{j}}}]|} \sum_{{\bf{j}} \in \mathcal{J}(k,n)}
\sum_{{\bf{i}} \in [{{\bf{j}}}]} | z_{{\bf{i}}}|^{r'} \\
& = \sup_{{\bf{j}} \in \mathcal{J}(k,n)} \frac{|c_{{\bf{j}}}|^{r'}}{|[{{\bf{j}}}]|} \sum_{{\bf{i}}\in \mathcal{M}(k,n)}
|z_{{\bf{i}}}|^{r'} = \sup_{{\bf{j}} \in \mathcal{J}(k,n)}\frac{|c_{{\bf{j}}}|^{r'}}{|[{{\bf{j}}}]|}\Big(\sum_{k=1}^n |z_k|^{r'}\Big)^{k}\,,
\end{align*}
which is exactly what we were looking for.
\end{proof}

The following result on homogeneous polynomials on $\ell_p^n$ is of special interest. Given $1 \leq p \leq \infty$, we use
the notation $r(p) := \max\{p',2\}$.

\begin{Coro} \label{analog}
Adopting the notation used in Remark~$\ref{notation}$, for every $1 < p \leq \infty$ there is a~constant $C_{r(p)} = C(p) >0$
such that, for every  $m$-homogeneous random polynomial $\sum_{|\alpha| = m} \gamma_\alpha  z^\alpha$ on $\ell^n_p$, we have
\[
\bigg\| \sup_{z \in B_{\ell^n_p}} \Big|\sum_{|\alpha|=m} \gamma_{\alpha} z^\alpha \Big|  \bigg\|_{L_{\varphi_{r(p)}}}
\leq C_{r(p)}  \big(n(1+\log m)\big)^{\frac{1}{r(p)}} n^{m \max\{\frac{1}{2}- \frac{1}{p},0\}}\,.
\]
In addition, fixing $m$ and assuming that the subgaussians $\gamma_\alpha$ are normal Gaussian, Rademacher or Steinhaus variables,
provided $ r(p) = 2$, and  Rademacher or Steinhaus variables, whenever $1\le r(p) < \infty$ is arbitrary, the preceding estimate
is asymptotically optimal in the sense that
\[
\bigg\| \sup_{z \in B_{\ell^n_p}} \Big|\sum_{|\alpha|=m} \gamma_{\alpha} z^\alpha \Big|  \bigg\|_{L_{\varphi_{r(p)}}} \asymp
n^{\frac{1}{r(p)} + m \max\{\frac{1}{2}- \frac{1}{p},0\}}\,,
\]
up to constants which only depend on $m$ and $p$ but not on $n$.
\end{Coro}

\begin{proof}
For the first statement we apply Corollary~\ref{KSZsix} to $r=r(p)$, and recall that by H\"older's inequality
\[
\sup_{z \in B_{\ell^n_p}} \big( \sum_{i=1}^n |z_i|^{r'}\big)^{1/r'} = n^{\max\{\frac{1}{r'}- \frac{1}{p},0\}}\,.
\]
To see the optimality in  the case of Rademacher random variables $\varepsilon_\alpha$, note first that
for every~$\omega$
\begin{align*}
\Big\|\sum_{k=1}^n z_k \Big\|_{B_{\ell_p^n}}^m & \leq
\Big\|\sum_{|\alpha|=m} \frac{m!}{\alpha!}\varepsilon_{\alpha}(\omega)  \varepsilon_{\alpha}(\omega) z^\alpha \Big\|_{B_{\ell_p^n}}\\
& \leq m! \,\chi(m,\ell_p^n) \Big\|\sum_{|\alpha|=m} \varepsilon_{\alpha}(\omega) z^\alpha \Big\|_{B_{\ell_p^n}}\,,
\end{align*}
where $\chi(m,\ell_p^n)$ stands for the unconditional basis constant of the basis sequence formed by all monomials $z^\alpha$
(for $\alpha \in \mathbb{N}_0^n$ with $|\alpha|=m$) in the Banach space $\mathcal{P}_m(\ell_p^n)$\,. But it is well-known that
\[
\chi(m,\ell_p^n) \asymp n^{(m-1)\big(1 - \frac{1}{\min\{p,2\}}\big)}\,,
\]
where the constants only depend on $m$ and $p$ (see, e.g., \cite[Corollary~19.8 ]{Defant})\,. This gives
\[
n^{m(1-\frac{1}{p})} n^{-(m-1) \big(1 - \frac{1}{\min\{p,2\}}\big)} \prec
\Big\|\sum_{|\alpha|=m} \varepsilon_{\alpha}(\omega) z^\alpha \Big\|_{B_{\ell_p^n}}\,.
\]
Taking norms in $L_{\varphi_{r(p)}}$ leads to the desired lower bound for Rademacher random variables. For Steinhaus and Gaussian
variables, note that in these cases $L_{\varphi_{r(p)}}$-averages are dominated by the corresponding Rademacher average.
\end{proof}

\subsection{Multilinear forms in Banach spaces}

\noindent
We here apply our techniques to spaces of multilinear forms on finite dimensional Banach spaces, and our main contribution is
as follows.

\begin{Theo} \label{KSZmulti}
Adopting the notation used in Remark~$\ref{notation}$, for  every $r \in [2, \infty)$ there is a~ constant $C_r >0$ such that,
for  every choice of finite dimensional $($real or complex$)$ Banach spaces $X_{j}$ with $\dim X_j = n_j, 1 \leq j \leq m$,
and $m$-linear mappings $L_1, \ldots, L_K \in \mathcal{L}_m(X_{1}, \ldots, X_{m})$, we have
\begin{align*}
&\bigg\|\sup_{(z_1, \ldots, z_m)\in B_{X_1 \times \cdots \times X_m}}
\Big|\sum_{i=1}^K  \gamma_i L_i (z_1, \ldots, z_m)\Big| \bigg\|_{L_{\varphi_r}} \\
& \leq C_r \Big(\sum_{j=1}^{m} n_j (1 + \log m)\Big)^{\frac{1}{r}}
\sup_{(z_1, \ldots, z_m)\in B_{X_1 \times \cdots \times X_m}}\big \|(L_i(z_1, \ldots, z_m))_{i=1}^K\big\|_{S_{r'}}\,.
\end{align*}
\end{Theo}

Our strategy for the proof is exactly as before, we start with a~multilinear analog of Lemma~\ref{bernstein2}.

\begin{Lemm}
\label{bernstein3}
Let $X_{j}$ with $\dim X_j = n_j, 1 \leq j \leq m$ be finite dimensional $($real or complex$)$ Banach spaces.
Then there is a~subset $F \subset \prod_{j=1}^mB_{X_{j}}$ of cardinality
\[
{\rm{card}}(F) \leq \prod_{j=1}^m \big(1+2m\big)^{2n_j}
\]
such that for every $L \in \mathcal{L}_m(X_{1}, \ldots, X_{m})$,
\[
\|L\|_\infty \leq 2 \sup_{(z_1, \ldots, z_m) \in F } |L(z_1, \ldots, z_m)|\,.
\]
\end{Lemm}

The proof is again based on Corollary \ref{Kballs}, hence, if all Banach spaces $X_i$ are real, we may replace the
exponents $2n_j$ by $n_j$.

\begin{proof}
For $L \in \mathcal{L}_m(X_{1}, \ldots, X_{m})$ and $z_j, v_j \in B_{X_j}$ for each $1 \leq j \leq m$, we have
\[
|L(z_1, \ldots, z_m) - L(v_1, \ldots, v_m)| \leq \sum_{k=1}^m L(z_1, \ldots, z_{k-1}, z_{k}-v_{k},v_{k+1}, \ldots, v_{m} )\,,
\]
and hence
\[
|L(z_1, \ldots, z_m) - L(v_1, \ldots, v_m)| \leq m \max_{1 \leq j \leq m} \,\|z_j - v_j\|\,\, \|L\|_\infty.
\]
By Corollary \ref{Kballs}, for each $1 \leq j \leq m$ there is $F_j \subset B_{X_j}$ with  $\text{card}\,F_j  \leq (1+2m)^{2n_j}$
such that
\[
B_{X_j} \subset \bigcup_{v \in F_j} B_{X_j}\Big(v, \frac{1}{2m}\Big)\,.
\]
Then for every $(z_1, \ldots, z_m)  \in B_{X_1} \times \cdots \times B_{X_m}$ there is some $(v_1, \ldots, v_m)
\in F:= F_1 \times \cdots \times F_m$ with $\max_{1 \leq j \leq m}\|z_j-v_j\| \leq \frac{1}{2m}$, and hence
\begin{align*}
&
|L(z_1, \ldots, z_m)| \\
& \leq |L(z_1, \ldots, z_m) - L(v_1, \ldots, v_m)| + |L(v_1, \ldots, v_m)| \leq \frac{1}{2}\|L\|_\infty +  \sup_{u \in F} |L(u)|\,.
\end{align*}
Since
$\text{card}\,F \leq \prod_{j=1}^m \big(1+2m\big)^{2n_j}$, the conclusion follows.
\end{proof}

\medskip

\begin{proof}[Proof of Theorem $\ref{KSZmulti}$] Consider the $2$-embed\-ding of  $E := \mathcal{L}_{m}(X_1, \ldots, X_m)$
in $\ell_\infty^N$  proved in  Lemma~\ref{bernstein3}. Then Theorem \ref{KSZmulti} is an immediate consequence of
Theorem~\ref{KSZtwo} (in the form given in Remark~\ref{sylvia})\,.
\end{proof}

The following immediate corollary extends Bayart's result from~\cite[Theorem~3.4]{Bayart}. Denote by $\mathcal{M}$ the union
of all index sets $\mathcal{M}(m,n) :=\{1, \ldots, n\}^{m}$ with $m, n \in \N$.

\begin{Coro}
\label{KSZfinal}
Using for the index set  $I = \mathcal{M}$ the notation of Remark~$\ref{notation}$, for every $r \in [2, \infty)$
there is a~constant $C_r >0$ such that for every $m$-linear random mapping
\[
\sum_{\mathfrak{j} =(j_1, \ldots, j_m) \in \mathfrak{J}} \gamma_\mathfrak{j}(\omega) \,c_\mathfrak{j}  \,z_1(j_1) \cdots z_m(j_m),
\quad\, \omega \in \Omega
\]
on the product $X_1 \times \cdots \times X_m$ of Banach spaces $X_{j} := (\mathbb{K}^{n_j}, \|\cdot\|_j), \,1\leq j\leq m$, where
$\mathfrak{J} := \prod_{j=1}^m\{1, \ldots, n_j\}$, we have
\begin{align*}
\bigg\| \sup_{(z_1, \ldots, z_m) \in  B_{X_1 \times \cdots \times X_m}}& \bigg| \sum_{\mathfrak{j} \in \mathfrak{J}}
\gamma_\mathfrak{j}\,c_\mathfrak{j}  \,z_1(j_1) \cdots z_m(j_m)\bigg| \bigg\|_{L_{\varphi_r}} \\
& \leq C_r \Big(\sum_{j=1}^{m} n_j (1 + \log m)\Big)^{\frac{1}{r}} \sup_{\mathfrak{j} \in \mathfrak{J}} |c_\mathfrak{j}|
\,\,\, \prod_{j=1}^m\sup_{z_j \in B_{X_j}} \Big( \sum_{k=1}^{n_j} |z_j(k)|^{r'}\Big)^{\frac{1}{r'} }\,.
\end{align*}
\end{Coro}

In the final part of this section we evaluate our results for the special case of $m$-linear mappings defined on products
$\ell_{p_1}^{n_1} \times \cdots \times \ell_{p_m}^{n_m}$. The results are multilinear versions of Corollary~\ref{analog}.
Given $\mathfrak{p} := (p_1, \ldots, p_m) \in [1, \infty]^m$,
\[
r(\mathfrak{p}) := \min\big\{ \max\{2,p_k'\}; 1 \leq k \leq m\big\} \in [2,\infty]\,.
\]
The following result was proved by Albuquerque and Rezende in  \cite[Proposition~2.3 and Theorem~2.4]{GurgelAlbuquerqueRezende}:
Assume that $m, n_1, \ldots, n_m \in \mathbb{N}$. Then there are signs $(\varepsilon_{\mathfrak{j}})_{\mathfrak{j}
\in \mathcal{M}}$, and an $m$-linear mapping $A$ on $\ell_{p_1}^{n_1} \times \cdots \times \ell_{p_m}^{n_m}$ given by
\[
A(z_1, \ldots, z_m):= \sum_{\mathfrak{j} \in \prod_{j=1}^m\{1, \ldots n_j\}} \varepsilon_\mathfrak{j} \,z_1(j_1) \cdots z_m(j_m)
\]
for all $(z_1, \ldots, z_m)\in \ell_{p_1}^{n_1} \times \cdots \times \ell_{p_m}^{n_m}$ such that
\begin{align}
\label{AR1}
\|A\| \leq C_m^{\frac{2}{r(\mathfrak{p})}} \Big(\sum_{j=1}^{m} n_j \Big)^{\frac{1}{r(\mathfrak{p})}}
\prod_{j=1}^m n_j^{\max\{\frac{1}{r(\mathfrak{p})'}-\frac{1}{p_j},0\}}\,,
\end{align}
where $C_m = (m!)^{1 -\max\{1/2, 1/\max\{p_1,\ldots,p_m\}\}} \sqrt{\log(1+4m)}$. In
\cite[Theorem~1.1]{PellegrinoSerranoSilva} this result was recently analysed  by Pellegrino, Serrano and Silva showing
that in fact, we may replace \eqref{AR1} by
\begin{align}
\label{AR2}
\|A\| \leq C_m^{\frac{2}{r(\mathfrak{p})}} \Big(\sum_{j=1}^{m} n_j \Big)^{\frac{1}{r(\mathfrak{p})}}
\prod_{j=1}^m n_j^{\max\{\frac{1}{2}-\frac{1}{p_j},0\}}\,,
\end{align}
an estimate which in the important case  $n= n_1 = \ldots = n_m$ for fixed $m$ turns out to be asymptotically correct in $n$.

All this is covered by the following more general result, where as before we let $\mathfrak{J} := \prod_{j=1}^m\{1, \ldots, n_j\}$.

\begin{Theo}
\label{Gurgel}
Using for the index set  $I = \mathcal{M}$ the notation of Remark~$\ref{notation}$, let $\mathfrak{p} = (p_1, \ldots, p_m)
\in [1, \infty]^m$, not all $p_j$'s equal $1$. Then there is a~constant $C_{r(\mathfrak{p})} =C(p_1, \ldots, p_m) >0$ such that
for every $m$-linear random  mapping $L$ on  $\ell_{p_1}^{n_1} \times \cdots \times \ell_{p_m}^{n_m}$ given by
\[
L(\omega, z_1, \ldots, z_m) := \sum_{\mathfrak{j} =(j_1, \ldots, j_m) \in \mathfrak{J}}
\gamma_\mathfrak{j}(\omega)\,c_{\mathfrak{j}}\,z_1(j_1) \cdots\, z_m(j_m), \quad\, \omega \in \Omega
\]
for all $(z_1, \ldots, z_m) \in \ell_{p_1}^{n_1} \times \cdots \times \ell_{p_m}^{n_m}$, we have
\begin{align*}
\bigg\|\sup_{(z_1,\ldots, z_m) \in B_{\ell_{p_1}^{n_1} \times \cdots \times \ell_{p_m}^{n_m}}} & \big| L(\cdot, z_1, \ldots, z_m)
\bigg\|_{L_{\varphi_{r(\mathfrak{p})}}} \\
& \leq C_{r(\mathfrak{p})} (1 + \log m)^{\frac{1}{r(\mathfrak{p})}} \sup_{\mathfrak{j} \in \mathfrak{J}} |c_\mathfrak{j}|\,\Big(\sum_{j=1}^{m} n_j \Big)^{\frac{1}{r(\mathfrak{p})}} \prod_{j=1}^m n_j^{\max\{\frac{1}{2}-\frac{1}{p_j},0\}}\,.
\end{align*}
In addition, assuming that $m$ is fixed,  $n= n_1 = \ldots = n_m$, and all  subgaussians $\gamma_\mathfrak{j}$ are normal Gaussian, Rademacher
or Steinhaus variables, provided $ r(\mathfrak{p}) = 2$, and  Rademacher or Steinhaus variables, whenever $1\le r(\mathfrak{p}) \le \infty$
is arbitrary, we have that the preceding estimate is optimal in the sense that
\begin{align*}
\bigg\|\sup_{\substack{z_j \in B_{\ell_{p_j}^{n_j} }\\1 \leq j \leq m}} \bigg|
\sum_{\mathfrak{j}  \in \mathfrak{J}} \gamma_\mathfrak{j}\,z_1(j_1) \cdots\, z_m(j_m)\bigg|
\bigg\|_{L_{\varphi_{r(\mathfrak{p})}}}
\asymp n^{\frac{1}{r(\mathfrak{p})} +\sum_{j=1}^m \max\{\frac{1}{2}- \frac{1}{p_j},0\}}\,,
\end{align*}
where the constants depend only on $m$ and the $p_j$'s but not on $n$.
\end{Theo}

\begin{proof}
First estimate: Recall that by  H\"older's inequality for each $j\in \{1, \ldots, m\}$
\[
\sup_{z_j \in B_{\ell_{p_j}^{n_j}}} \Big( \sum_{k=1}^{n_j} |z_j(k)|^{r(\mathfrak{p})'}\Big)^{\frac{1}{r(\mathfrak{p})'}}
=n_j^{\max\{\frac{1}{r(\mathfrak{p})'}-\frac{1}{p_j},0\}}\,.
\]
Now apply Corollary~\ref{KSZfinal} with  $r = r(\mathfrak{p}) \in [2, \infty)$ to get that
\begin{align*}
\bigg\|\sup_{(z_1,\ldots, z_m) \in B_{\ell_{p_1}^{n_1} \times \cdots \times \ell_{p_m}^{n_m}}} & \big| L(\cdot, z_1, \ldots, z_m)\big|
\bigg\|_{L_{\varphi_{r(\mathfrak{p})}}} \\
& \leq C_{r(\mathfrak{p})} (1 + \log m)^{\frac{1}{r(\mathfrak{p})}} \sup_{\mathfrak{j} \in \in \mathfrak{J}} |c_\mathfrak{j}|\,\Big(\sum_{j=1}^{m} n_j \Big)^{\frac{1}{r(\mathfrak{p})}} \prod_{j=1}^m n_j^{\max\{\frac{1}{r(\mathfrak{p})'}-\frac{1}{p_j},0\}}\,.
\end{align*}
This leads to the desired estimate. Indeed, without loss of generality we assume that $p_1 \ldots \leq p_m$, and hence we
consider the three cases
\begin{itemize}
\item[{\rm(1)}] either $p_m \geq 2$\,,
\item[{\rm(2)}] or $2 \leq p_m$\,,
\item[{\rm(3)}] or $p_1 \leq \ldots p_d < 2 \leq p_{d+1} \leq \ldots \leq p_m$.
\end{itemize}
In case $(1)$ we have that $r(\mathfrak{p})=2$, and the result follows. In case $(2)$ we have that
that $r(\mathfrak{p}) \geq p_j'$ for all $1 \leq j \leq \leq m$, hence
\[
\prod_{j=1}^m n_j^{\max\{\frac{1}{r(\mathfrak{p})'}-\frac{1}{p_j},0\}} = 1 = \prod_{j=1}^m n_j^{\max\{\frac{1}{2}-\frac{1}{p_j},0\}}\,,
\]
and we again get what we want. It remains to handel case $(3)$: Note first that in this case
$r(\mathfrak{p})=2$. Moreover, for all $\omega$
\begin{align*}
f(\omega):=\Big\|L(\omega,\cdot)\colon & \prod_{j=1}^d\ell_{p_j}^n\times \prod_{j=d+1}^m\ell_{p_j}^n \to \mathbb{K}\Big\| \\
& \leq  g(\omega):=\Big\|L(\omega,\cdot)\colon \prod_{j=1}^d\ell_{p_j}^n\times \prod_{j=d+1}^m\ell_{2}^n \to \mathbb{K}\Big\|, \quad\,
\omega \in \Omega\,,
\end{align*}
and, since $r(p_1, \ldots,p_d, 2, \ldots, 2) = 2$,  we obtain
\begin{align*}
\big\|f(\omega) \big\|_{L_{\varphi_{2}}} & \leq \big\|g(\omega) \big\|_{L_{\varphi_{2}}} \\
& \leq  C_{(p_1, \ldots,p_d, 2, \ldots, 2)}  (1 + \log m)^{\frac{1}{2}} \,\Big(\sum_{j=1}^{m} n_j \Big)^{\frac{1}{2}}
\prod_{j=1}^m n_j^{\max\{\frac{1}{2}-\frac{1}{p_j},0\}}.
\end{align*}
\noindent Second estimate:
Let us first look at Rademacher variables $\varepsilon_{\mathfrak{j}}$. Then it is proved in  \cite[Section 2.2.]{PellegrinoSerranoSilva}
that for all unimodular $m$-linear forms given by
\[
L(\omega, z_1, \ldots, z_m) := \sum_{\mathfrak{j}\in \mathcal{M}(m,n)} \varepsilon_\mathfrak{j}(\omega)\,z_1(j_1) \cdots z_m(j_m), \quad\,
\omega \in \Omega\,,
\]
we have that
\[
\Big\|L(\omega,\cdot)\colon  \prod_{j=1}^m\ell_{p_j}^n \to \mathbb{K}\Big\| \geq D_m n^{\frac{1}{r(\mathfrak{p})}
+\sum_{j=1}^m \max\{\frac{1}{2}- \frac{1}{p_j},0\}}\,,
\]
where the constant $D_m >0$ only depends on $m$. Taking norms in $L_{\varphi_{r(\mathfrak{p})}}$, finishes the argument for this case. But
vector-valued  $L_{\varphi_{r(\mathfrak{p})}}$-averages taken with respect to Steinhaus or Gaussian  random variables  dominate the corresponding $L_{\varphi_{r(\mathfrak{p})}}$-averages for Rademacher random variables which completes the argument.
\end{proof}

\section{$K\!S\!Z$--type inequalities via interpolation}

In this section we use interpolation theory to prove more 'abstract $K\!S\!Z$--inequalities' in the sense of
\eqref{abstract}, which in fact  extend and strengthen some of our previous results.

\subsection{Exact interpolation functors}

Let $\mathcal{F}$ be an exact interpolation functor. In what follows we use an inequality that is an obvious
consequence of the definition of the fundamental function $\phi_{\mathcal{F}}$ (given in the preliminaries):
For any operator $T\colon \xo \to \yo$ between Banach couples $\xo = (X_0, X_1)$ and $\yo=(Y_0, Y_1)$, we have
\[
\big\|T\colon \mathcal{F}(\xo) \to \mathcal{F}(\yo)\big\|
\leq \phi_{\mathcal{F}} \big(\|T\colon X_0 \to Y_0\|, \, \|T\colon X_1\to Y_1\|\big)\,.
\]
In this section we mainly consider a~special class of exact interpolation functors $\mathcal{F}$.
Clearly, by the interpolation property, it follows that for any Banach couple $\xo=(X_0, X_1)$, we have
\[
\sup_{N\geq 1} \big\|\text{id}\colon \mathcal{F}(\ell_\infty^N(X_0), \ell_\infty^N(X_1))
\hookrightarrow \ell_\infty^N(\mathcal{F}(X_0, X_1))\big\| \leq 1\,.
\]
This motivates us to introduce the following definition: An interpolation functor $\mathcal{F}$
is said to have the $\infty$-property on $\xo$ with constant $\delta>0$ whenever
\[
\sup_{N\geq 1} \big\|\text{id}\colon  \ell_\infty^N(\mathcal{F}(X_0, X_1))
\hookrightarrow \mathcal{F}(\ell_\infty^N(X_0), \ell_\infty^N(X_1))\big\| \leq \delta\,,
\]
and $\mathcal{F}$ has the uniform $\infty$-property with constant $\delta$ whenever it has the $\infty$-property
on any Banach couple $\xo$ with constant $\delta$.

Moreover, we need the following useful interpolation formula from \cite{Bukhvalov}, which is a~consequence
of the Hahn--Banach--Kantorovich theorem.

\begin{Lemm}
\label{intmixA}
Let $E_0$ and $E_1$ be Banach function lattices on a measure space $(\Omega, \mathcal{A}, \mu)$ and let $X$
be a~Banach space. Then, for any exact interpolation functor $\mathcal{F}$, we have
\[
\mathcal{F}(E_0(X), E_1(X)) \cong  \mathcal{F}(E_0, E_1)(X)\,.
\]
\end{Lemm}

Now we are prepared to prove the following key interpolation theorem based on the case $r=2$ from
Theorem~\ref{matrix}. The space of all scalar $N\times K$-matrices is denoted by $\mathcal{M}_{N, K}$.

\begin{Theo}
\label{intmixB}
Let $(\gamma_i)_{i \in \N}$ be a sequence of $($real or complex$)$ subgaussian random variables such that
$s = \sup_i \text{sg}(\gamma_i) < \infty$ and $M = \sup_i\|\gamma_i\|_\infty < \infty$. Suppose that
$\mathcal{F}$ is an exact interpolation functor with the $\infty$-property with constant $\delta$.

Then there exists a~constant $C=C(s,M)>0$ such that for every matrix $(a_{i, j}) \in \mathcal{M}_{N, K}$, we have
\[
\Big\|\sup_{1 \leq j \leq N} \big|\sum_{i=1}^K \gamma_i a_{i,j}\big|\Big\|_{\mathcal{F}(L_\infty, L_{\varphi_2})}
\leq \delta C \phi_{\mathcal{F}}\big(1, \sqrt{1 + \log N}\,\big)
\sup_{1\leq j\leq N} \big\|(a_{i,j})_{i=1}^K\big\|_{\mathcal{F}(\ell_1, \ell_2)}\,,
\]
where $\phi_{\mathcal{F}}$ is the fundamental function of $\mathcal{F}$. In particular, $C= \sqrt{8/3}$,
whenever $(\gamma_i)= (\varepsilon_i)$ is a sequence of independent random Rademacher variables.
\end{Theo}

\begin{proof}
Define the linear mapping $T\colon \mathcal{M}_{N, K} \to L^0(\mathbb{P}, \ell_\infty^N)$ by
\[
T(a_{i, j}):= \Big(\sum_{i=1}^K \gamma_i a_{i, j}\Big)_{j=1}^N, \quad\, (a_{i, j})\in \mathcal{M}_{N, K}\,.
\]
We claim that
\[
T\colon \big(\ell_\infty^N(\ell_{1}^K), \ell_\infty^N(\ell_{2}^K)\big) \to
\big(L_{\infty}(\ell_{\infty}^N), L_{\varphi_2}(\ell_{\infty}^N)\big)\,.
\]
Obviously, $T\colon \ell_\infty^N(\ell_{1}^K) \to L_{\infty}(\ell_{\infty}^N)$ with norm $\|T\|\leq \sup_i \|\gamma_i\|_{\infty}$.
From Theorem \ref{matrix}, it follows that $T\colon \ell_\infty^N(\ell_{2}^K) \to L_{\varphi_2}(\ell_{\infty}^N)$
has norm less than or equal to $C_2 \sqrt{1+ \log N}$. By the interpolation property, and our hypothesis that $\mathcal{F}$
has the $\infty$-property with constant $\delta$, we get that for all $(a_{i,j}) \in \mathcal{M}_{N,K}$
\[
\big\|T(a_{i,j})\big\|_{\mathcal{F}(L_{\infty}(\ell_{\infty}^N), L_{\varphi_2}(\ell_{\infty}^N))} \leq
C\,\delta\, \phi_{\mathcal{F}}(1, \sqrt{1 + \log N}) \big\|(a_{i,j})\big\|_{\ell_{\infty}^N(\mathcal{F}(\ell_1^K, \ell_2^K))}\,.
\]
Since $(\ell_1^K, \ell_2^K)$ is a $1$-complemented sub-couple of the couple $(\ell_1, \ell_2)$,
\[
\mathcal{F}(\ell_1^K, \ell_2^K) \cong \mathcal{F}(\ell_1, \ell_2)^K\,.
\]
Thus the above interpolation estimate combined with Lemma \ref{intmixA} yields the required estimate. If $(\gamma_i)=
(\varepsilon_i)_{i \in \N}$, then we have $\|T\|=1$ and $C_2 = \sqrt{8/3}$.
\end{proof}

As an application of Theorem \ref{intmixB}, we get the interpolation variant of Theorem~\ref{matrix}
(as in the form given in Remark~\ref{sylvia}).

\begin{Rema}
\label{sylvia2}
Let $(\gamma_i)_{i \in \N}$ be a sequence of $($real or complex$)$ subgaussian random variables such that
$s = \sup_i \text{sg}(\gamma_i) < \infty$ and $M = \sup_i\|\gamma_i\|_\infty < \infty$. Let $\mathcal{F}$ be an
exact interpolation functor with the $\infty$-property with constant $\delta$.

Then there exists a~constant $C=C(s,M)>0$  such that, for every Banach space $E$, every $\lambda$-embedding
$I\colon  E \hookrightarrow \ell_\infty^N$, and every choice of $x_1, \ldots, x_K \in E$, we have
\begin{align*}
\bigg\|\sum_{i=1}^K  \gamma_i x_i\bigg\|_{\mathcal{F}(L_\infty, L_{\varphi_2})(E)} \leq   \|I^{-1}\| C
\delta \,\phi_{\mathcal{F}}\big(1, \sqrt{1 + \log N}\,\big)
\sup_{1 \leq j \leq N} \big\|\big(I(x_i)(j)\big)_{i=1}^K\big\|_{\mathcal{F}(\ell_1, \ell_2)}\,,
\end{align*}
where $\phi_{\mathcal{F}}$ is the fundamental function of $\mathcal{F}$.
\end{Rema}

In order to apply all this within  the setting of Orlicz spaces, the following lemma from \cite[Lemma 3]{MastyloSzwedek}
is going to be crucial.

\begin{Lemm}
\label{appl}
Let $\mathcal{F}$ be an exact interpolation functor with  characteristic function $\psi=\psi_{\mathcal{F}}$.
Then the following embedding
\[
\mathcal{F}(L_\infty, L_{\varphi_2}) \hookrightarrow L_\Phi(\mathbb{P})
\]
is contractive, where $\Phi$ and $\Psi$  are Orlicz functions satisfying for all $t >0$
\[
\Phi(t)= e^{\Psi(t)}-1\,\,\, \,\,\text{and} \,\,\,\,\, \Psi^{-1}(t)\asymp \psi_{*}(1, \sqrt{t})\,.
\]
\end{Lemm}

\subsection{The $K$--method}

We specialize the above results to some interpolation methods which play a~fundamental role in interpolation theory,
namely the $K$-method and the Orbit method. In order to recover the random inequalities from our results above,
the main difficulty lies in proving that the given  exact interpolation functor $\mathcal{F}$ has the $\infty$-property
with some constant $\delta$, and we also need to know the best possible estimate of the fundamental function of
$\phi_{\mathcal{F}}$. It should be pointed out here that the key Theorem~\ref{intmixB} shows that in fact, we only need
to know that $\mathcal{F}$ has the $\infty$-property on the special Banach couple $(L_\infty, L_{\varphi_2})$.

We start with the $K$--method of interpolation. Let $\Phi$ be a~Banach sequence lattice of (two-sided) sequences
such that $(\min\{1, 2^k\})_{k\in \mathbb{Z}} \in F$. If $(X_0, X_1)$ is a~Banach couple, then the $K$-method of
interpolation produces $(X_0, X_1)_F$, the Banach space of all $x\in X_0 + X_1$ equipped with the norm
\[
\|x\|:= \big\|\big(K(1, 2^k, x; X_0, X_1)\big)_k\big\|_F \,,
\]
where $K$ is the Peetre functional given for all $x\in X_0 + X_1$ and all $s, t>0$ by
\[
K(s, t, x; X_0, X_1):= \inf\big\{s\|x_0\|_{X_0} + t\|x_1\|_{X_1};\, x= x_0 + x_1,\, x_0 \in X_0,\, x_1\in X_1\big\}\,.
\]
If $\psi \in \mathcal{Q}$ (see the preliminaries) and $F:= \ell_{\infty}(1/\psi(1, 2^{n}))$, then the space $(X_0, X_1)_\Phi$
is denoted by $(X_0, X_1)_{\psi, \infty}$. In the particular case that $\theta \in (0, 1)$ and $\psi(s, t)= s^{1-\theta}t^{\theta}$
for all $s, t>0$, we recover the classical Lions--Peetre space $(X_0, X_1)_{\theta, \infty}$.

In what follows, for any $\psi\in \mathcal{Q}$, we define the function $\overline{\psi} \in \mathcal{Q}$ by
\begin{align*}
\overline{\psi}(s,t) = \sup\bigg\{\frac{\psi(us, vt)}{\psi(u,v)};\, \quad\, u, v>0\bigg\}, \quad\, s, t>0\,.
\end{align*}
Moreover, we need another lemma.

\begin{Lemm}
For any $\psi\in \mathcal{Q}$, the exact interpolation functor $\mathcal{F}:=(\,\cdot\,)_{\psi, \infty}$ has the $\infty$-property
with constant $2$ and its fundamental function satisfies $\phi_{\mathcal{F}} \leq 2\,\overline{\psi}$.
\end{Lemm}

\begin{proof}
Fix Banach couples $\xo =(X_0, X_1)$ and $\yo =(Y_0, Y_1)$. Routine calculations show that, for all $(x_j)_{j=1}^N$ in $X_0 + X_1$,
we have
\begin{align*}
\max_{1\leq j\leq N} K(1,t, x_j; X_0, X_1) & \leq K\big(1,t, (x_j)_{j=1}^N; \ell_\infty^N(X_0), \ell_\infty^N(X_0)\big) \\
& \leq 2\,\max_{1\leq j\leq N} K(1,t, x_j; X_0, X_1)\,.
\end{align*}

This immediately implies that $\mathcal{F}$ has $\infty$ property with constant $2$. Since for any operator
$T\colon \xo  \to \yo$, $x\in X_0 + X_1$, and $ n\in \mathbb{Z}$
\[
K(1,2^n, Tx; \yo) \leq K(\|T\colon X_0 \to Y_0\|, 2^n \|T\colon X_1 \to Y_1\|, Tx; \xo)\,,
\]
the estimate $\phi_{\mathcal{F}} (s, t) \leq 2 \overline{\psi}(s, t)$ for all $s, t>0$ is obvious.
\end{proof}

Then for the special case of Lions-Peetre interpolation the following consequence is  immediate from Theorem~\ref{intmixB}
(in the form given in Remark~\ref{sylvia2}).

\begin{Coro}
\label{marcinkiewicz}
Let $\psi \in \mathcal{Q}$, and $(\gamma_i)_{i \in \N}$ be a sequence of $($real or complex$)$ subgaussian random variables
such that $s = \sup_i \text{sg}(\gamma_i) < \infty$ and $M = \sup_i\|\gamma_i\|_\infty < \infty$.

Then there exists a~constant $C=C(s,M)>0$ such that, for every Banach space $E$, every $\lambda$-embedding
$I\colon E \hookrightarrow \ell_\infty^N$, and every choice of $x_1, \ldots, x_K \in E$, we have
\begin{align*}
\bigg\|\sum_{i=1}^K  \gamma_i x_i\bigg\|_{(L_\infty, L_{\varphi_2})_{\psi, \infty}(E)} \leq 2 \lambda C\,
\overline{\psi}\big(1, \sqrt{1 + \log N}\,\big)
\sup_{1 \leq j \leq N} \big\|\big(I(x_i)(j)\big)_{i=1}^K\big\|_{\psi, \infty}\,.
\end{align*}
\end{Coro}

This fact combined with Lemma \ref{appl}, recovers Theorem~\ref{matrix} (and Remark~\ref{sylvia}) in  the
case $2<r<\infty$. Indeed, to see this we use a well-known interpolation formula, which states that for all
$1\leq p_0<p_1<\infty$ and $\theta \in (0, 1)$, we have,
\[
(\ell_{p_0}, \ell_{p_1})_{\theta, \infty} = \ell_{p, \infty}\,,
\]
where $1/p = (1-\theta)/p_0 + \theta/p_1$ (see \cite[Theorem 5.2.1]{BL}). Thus if $2<r <\infty$, then the
above formula yields, with $\theta = 2/r$ that
\[
(\ell_1, \ell_2)_{\theta, \infty} = \ell_{r', \infty}\,,
\]
where $1/r + 1/r'=1$. It is easily checked that for $\mathcal{F} = (\,\cdot\,)_{\theta, \infty}$, we have that
$\psi_{\mathcal{F}}(s, t) = s^{1-\theta}t^{\theta}$ and  $\phi_{\mathcal{F}}(s, t)
\leq s^{1-\theta}t^{\theta}$ for all $s, t>0$.

Now observe that if $\theta = 2/r$, then  the Orlicz function $\Psi$ which satisfies $\Psi^{-1}(t)  =
\psi_{\mathcal{F}}(1, \sqrt{t})$ is given by $\Psi(t) = t^r$ for all $t>0$ and so $\Phi(t)
:= e^{\Psi(t)}- 1 = e^{t^r} - 1$ for all $t\geq 0$. Since the functor $(\,\cdot\,)_{\theta, \infty}$ has the
$\infty$-property, Lemma~\ref{appl} applies and so we recover Theorem~\ref{matrix} (and Remark~\ref{sylvia}).

We refer to \cite{Mastylo}, where it shown that for some class of functions $\psi$, the interpolation spaces
$(\ell_1, \ell_2)_{\psi, \infty}$ equal, up to equivalence of norms, the Marcinkiewicz symmetric sequence spaces $m_{w}$,
where the weight $w=(w_n)$ only depends on $\psi$. Moreover, these results show that in the scalar case the estimate in
Corollary \ref{marcinkiewicz} is best possible in general, that is, the two sides of the inequality appearing there are
equivalent.

\subsection{The orbit method}

Now we consider the method of orbits (see \cite{BK, Ov84}). Given a Banach couple $\ao = (A_0, A_1)$, we fix an
arbitrary element $a\neq 0$ in $A_0 + A_1$. The \emph{orbit} of the  element $a$ in a~Banach couple $\xo$
is the Banach space $\text{Orb}_{\ao}(a, \cdot):=\{Ta; \, T\colon \ao \to \xo\}$ equipped with the norm
\[
\|x\| := \inf\big\{\|T\colon \ao \to \xo\|; \, T\colon \ao \to \xo,\, x=Ta\big\}\,.
\]
It is easy to see that $\mathcal{F}:=\text{Orb}_{\ao}(a, \cdot)$ is an exact interpolation functor. The fundamental
function $\phi_{\mathcal{F}}$ of $\mathcal{F}$ is given by the formula (see \cite[p.~389--390]{Ov84})
\[
\phi_{\mathcal{F}}(s, t) = 1/K(s^{-1}, t^{-1}, a; \ao), \quad\, s, t>0\,.
\]
We need the following lemma.

\begin{Lemm}
\label{orbit}
Given a Banach couple $\vec{A}= (A_0, A_1)$ and $a\neq 0$ in $A_0 + A_1$. Then, for any Banach couple $\xo=(X_0, X_1)$
and each positive integer $N$, we have with $\mathcal{F}:={\rm{Orb}}_{\ao}(a, \cdot)$
\[
\ell_{\infty}^N(\mathcal{F}(X_0, X_1)) \cong \mathcal{F}(\ell_{\infty}^N(X_0), \ell_{\infty}^N(X_1)), \quad\, N\in \mathbb{N}\,,
\]
that is, the functor $\text{Orb}_{\ao}(a, \cdot)$ has the $\infty$-property with constant $1$.
\end{Lemm}

\begin{proof} It is enough to show that, for each $N\in \mathbb{N}$,
\[
\big\|\text{id}\colon  \ell_{\infty}^N(\mathcal{F}(X_0, X_1))
\hookrightarrow \mathcal{F}(\ell_{\infty}^N(X_0), \ell_{\infty}(X_1))\big\| \leq 1\,.
\]
Fix $x=(x_j)_{j=1}^N \in \ell_{\infty}^N(\mathcal{F}(\xo))$ with $\|x\|_{\ell_{\infty}^N(\mathcal{F}(\xo))} \leq 1$.
This implies that, for each $1\leq j\leq N$ there exists $T_j\colon \ao \to \xo$ such that $x_j = T_j(a)$ and
$\|T\colon \ao \to \xo\| \leq 1$.
Define an operator $\oplus\,T_j \colon A_0 + A_1 \to \ell_{\infty}^N(X_0) + \ell_{\infty}^N(X_1)$, by
\[
\oplus\,T_j(b):= (T_{j}b)_{j=1}^N, \quad\, b\in A_0 + A_1\,.
\]
Observe that $\oplus\,T_j \colon (A_0, A_1) \to (\ell_{\infty}^N(X_0), \ell_{\infty}^N(X_1))$ with
\[
\big\|\oplus\,T_j \colon A_i \to \ell_{\infty}^N(X_i)\| =
\sup_{1\leq j\leq N} \|T_j\colon A_i \to X_i\| \leq 1, \quad\, i=0, 1\,.
\]
Since $\oplus\,T_j(a) = (T_ja)_{j=1}^N = (x_j)_{j=1}^N = x$, it follows that
\[
x\in \mathcal{F}(\ell_{\infty}^N(X_0), \ell_{\infty}(X_1))
\]
with $\|x\|\leq 1$. This completes the proof.
\end{proof}

For a given $\varphi \in \mathcal{Q}$, we let $a_{\varphi}: = (\varphi(1, 2^n))_{n \in \mathbb{Z}}$ and
$\vec{\ell}_{\infty}:= (\ell_\infty, \ell_\infty(2^{-n}))$. We consider the orbit $\text{Orb}_{\ell_\infty}(a_{\varphi}, \cdot)$,
and remark that this functor appeared in \cite{Ov76} in a~slightly different form. In what follows this functor is denoted by
$\varphi_{\ell}$.

In order to make applications of our above result to the interpolation functor $\varphi_{\ell}$, we need to estimate the
fundamental function of this functor. Thus we provide a~close to optimal estimate which surely could be useful in other types
of interpolation problems.

\begin{Lemm}
\label{fundament}
If $\varphi\in \mathcal{Q}$, then for every operator $T\colon (X_0, X_1) \to (Y_0, Y_1)$ between Banach couples, we have
\[
\|T\colon \varphi_\ell(X_0, X_1) \to \varphi_\ell(Y_0, Y_1)\| \leq 4\,\overline{\varphi}\,(\|T\colon X_0 \to Y_0\|, \,
\|T\colon X_1 \to Y_1\|)\,,
\]
that is, the fundamental function  of $\varphi_{\ell}$ satisfies $\phi_{\varphi_\ell} \leq 4 \overline{\varphi}$.
\end{Lemm}

\begin{proof}
For the proof we will need the isometrical formula $\text{Orb}_{\ao}(a, \vec{\ell}_\infty) \cong (\vec{\ell}_\infty)_{\psi, \infty}$,
where $\psi(s, t) := K(s, t, a; \ao)$ for all $s, t>0$. We first prove a major step:
\[
\big\|\text{id}\colon (\vec{\ell}_\infty)_{\psi, \infty} \hookrightarrow \text{Orb}_{\ao}(a, \vec{\ell}_\infty)\big\| \leq 1\,.
\]
Fix $\xi := (\xi_n)\in (\ell_\infty, \ell_\infty(2^{-n}))_{\psi, \infty}$ with $\|\xi\|_{\psi, \infty} \leq 1$. Then
\[
K(2^n, \xi; \ell_\infty) \leq K(2^n, a; \ao), \quad\, n\in \mathbb{Z}\,.
\]
By the Hahn--Banach theorem, for each $n\in \mathbb{Z}$ we can find a~functional $f_n \in (A_0 + A_1)^{*}$ such that
$f_n(a)= K(2^n, a; \ao)$ and
\[
|f_n(x)| \leq K(2^n, x; \ao), \quad\, x\in A_0 + A_1\,.
\]
This inequality implies that $\sup_{n\in \mathbb{Z}} \|f_n\|_{A_0^{*}} \leq 1$ and
$\sup_{n\in \mathbb{Z}} 2^{-n} \|f_n\|_{A_1^{*}} \leq 1$. It is easy to see that
\[
|\xi_n| \leq K(2^n, \xi; \vec{\ell}_\infty), \quad\, n\in \mathbb{Z}\,.
\]
From the above relations, we conclude that the mapping $S$ given on $A_0 + A_1$ by the formula
\[
Sx := \Big\{\frac{\xi_n}{K(2^n, a; a)}f_n(x)\Big\}_{n\in \mathbb{Z}}\,, \quad\, x\in A_0 + A_1\,,
\]
defines a bounded operator from $\ao$ into $\vec{\ell}_\infty$ with $\|S\colon \ao \to \vec{\ell}_\infty\| \leq 1$
and $Sa= \xi$. In consequence $\xi \in \text{Orb}_{\ao}(a, \vec{\ell}_\infty)$ with $\|\xi\|_{\text{Orb}} \leq 1$.
This proves the major step.

Since, for any operator $T\colon \ao \to \vec{\ell}_\infty$,
\[
K(2^n, Ta; \vec{\ell}_\infty) \leq \|T\colon \ao \to \vec{\ell}_\infty\|\,K(2^n, a; \ao), \quad\, n\in \mathbb{Z}\,,
\]
the reverse continuous inclusion follows with
\[
\big\|\text{id}\colon \text{Orb}_{\ao}(a, \vec{\ell}_\infty) \to (\ell_\infty, \ell_\infty(2^{-n}))_{\psi, \infty}\big\| \leq 1\,.
\]
Now we will use the isometrical formula shown above with $\ao: = \vec{\ell}_\infty$ to get that for
$\varphi_{\ell}(\vec{\ell}_\infty):=\text{Orb}_{\vec{\ell}_\infty}(a_\varphi, \vec{\ell}_\infty)= \ell_{\infty}\big(\frac{1}{\varphi(1, 2^n)}\big)$ with
\[
\frac{1}{2}\,\sup_{n\in \mathbb{Z}} \frac{|\xi_n|}{\varphi(1, 2^n)} \leq \|\xi\|_{\varphi_{\ell}(\vec{\ell}_\infty)} \leq 2\,\sup_{n\in \mathbb{Z}} \frac{|\xi_n|}{\varphi(1, 2^n)}\,.
\]
To see this we recall the following easily verified formula which states that, for all $\xi=(\xi_k) \in \ell_\infty + \ell_\infty(2^{-k})$,
we have
\[
\big\|\big(\min\{s, 2^{-k} t\}\xi_k\big)\|_{\ell_\infty} \leq K(s, t, \xi; \vec{\ell}_\infty) \leq 2
\big\|\big(\min\{s, 2^{-k} t\}\xi_k\big)\big\|_{\ell_\infty},
\quad\, s, t>0\,.
\]
In particular,  we get that for $\psi(1, 2^n):= K(2^n, a_\varphi, \vec{\ell}_\infty)$ with $a_\varphi=\{\varphi(1, 2^k)\}$
the following estimates hold:
\[
\varphi(1, 2^n) \leq \psi(1, 2^n) \leq 2 \sup_{k\in \mathbb{Z}} \min\Big\{1, \frac{2^n}{2^k}\Big\}\varphi(1, 2^k) = 2 \varphi(1, 2^n),
\quad\, n\in \mathbb{Z}\,.
\]
Now we are ready to prove the required statement. Let $T\colon \xo \to \yo$ be a~nontrivial operator. Fix $x\in \varphi_{\ell}(\xo)$,
and take any $S\colon \vec{\ell}_\infty \to \xo$ such that $x= Sa_\varphi$.

For each $\nu\in \mathbb{Z}$, we consider the shift operator $\tau_\nu$ defined by $\tau_{\nu}(\xi_n) := (\xi_{n+\nu})$. Clearly,
$\tau_{\nu} \colon \vec{\ell}_\infty \to \vec{\ell}_\infty$ with $\|\tau_{\nu} \colon \vec{\ell}_\infty \to \vec{\ell}_\infty\| =
\max\{1, 2^{\nu}\}$. By the interpolation property $Tx\in \varphi_{\ell}(\yo)$. Then, for each $k\in \mathbb{Z}$, we get that
\begin{align*}
& \|Tx\|_{\varphi_{\ell}(\yo)} = \|T(Sa_\varphi)\|_{\varphi_{\ell}(\yo)} = \|TS\tau_{-k}(\tau_{k} a_\varphi)\|_{\varphi_{\ell}(\yo)}\\
& \leq \max\big\{\|T S\tau_{-k}\colon \ell_\infty \to Y_0\|, \,
\|T S\tau_{-k}\colon \ell_\infty(2^{-n}) \to Y_1\|\big\}\,\|\tau_{k}a_{\varphi}\|_{\varphi_{\ell}(\vec{\ell}_\infty)}\\
& \leq \max\big\{\|T\colon X_0 \to Y_0\|, \, 2^{-k} \|T\colon X_1 \to Y_1\|\big\}\,\big\|\big(\varphi(1, 2^{n+k}\big)\big\|_{\varphi_{\ell}(\vec{\ell}_\infty)}\, \|S\colon \vec{\ell}_\infty \to \xo\|\,.
\end{align*}
Choose $k$ such that $2^k \|T\colon X_0 \to Y_0\| \leq \|T\colon X_1 \to Y_1\| < 2^{k+1}  \|T\colon X_0 \to Y_0\|$. Then applying
the estimate proved  above, we obtain
\begin{align*}
\big\|\{\varphi(1, 2^{n+k}\}\big\|_{\varphi_{\ell}(\vec{\ell}_\infty)} & \leq 2\,\sup_{n\in \mathbb{Z}} \frac{\varphi(1, 2^{n +k})}{\varphi(1, 2^n)}
\leq 2\,\overline{\varphi}(1, 2^k) \\
& \leq 2\,\overline{\varphi} \big(1, \|T\colon X_1\to Y_1\|/|T\colon X_0\to Y_0\|\big)\,.
\end{align*}
Since $S\colon \vec{\ell}_\infty \to \xo$ with $x= Sa_\varphi$ was arbitrary, the above estimates yields
\begin{align*}
\|Tx\|_{\varphi_{\ell}(\yo)} & \leq 4\,\|T\colon X_0 \to Y_0\| \,\overline{\varphi}\,\big(1,\, \|T\colon X_1 \to Y_1\|/\|T\colon X_0\to Y_0\|\big) \|x\|_{\varphi_{\ell}(\yo)} \\
& = 4\,\overline{\varphi} \big(\|T\colon X_0\to Y_0\|,\, \|T\colon X_1\to Y_1\|\big) \|x\|_{\varphi_{\ell}(\xo)}\,.
\end{align*}
This completes the proof.
\end{proof}

\subsection{The Calder\'on--Lozanovskii method}

It is well known that if $\varphi\colon \mathbb{R}_{+} \times \mathbb{R}_{+} \to \mathbb{R}_{+}$ is a~non-vanishing,
concave function, which is continuous in each variable and positive homogeneous of degree one, then for any couple
$(X_0, X_1)$ of Banach function lattices on a~measure space $(\Omega, \mathcal{A}, \mu)$ with the Fatou property the
formula
\[
\varphi_\ell(X_0, X_1) = \varphi(X_0, X_1)
\]
holds (see \cite{Ov76}), up to equivalence of norms with universal constants. Here  $\varphi(X_0, X_1)$ denotes the
Calder\'on--Lozanovskii space, which consists of all $f\in L^{0}(\mu)$ such that
$|f| \leq \lambda \,\varphi(|f_0|, |f_1|)$ $\mu$-a.e. for some
$\lambda > 0$ and $f_j \in B_{X_j}$, $j\in\{0, 1\}$. It is a~Banach lattice endowed with the norm
\[
\|f\| := \inf \big\{\lambda >0; \,|f| \leq \lambda \varphi(|f_0|, |f_1|),\, \|f_0\|_{X_0} \leq 1, \,\|f_1\|_{X_1} \leq 1\big\}.
\]
Combining the Lemmas \ref{orbit} and \ref{fundament} with Remark \ref{sylvia2} yields the following result.

\begin{Theo}
Let $(\gamma_i)_{i \in \N}$ be a sequence of $($real or complex$)$ subgaussian random variables such that
$s = \sup_i \text{sg}(\gamma_i) < \infty$ and $M = \sup_i\|\gamma_i\|_\infty < \infty$. Assume that $\varphi \in \mathcal{Q}$
is a~concave function.

Then there is a~universal constant $c >0$ and a~constant $C=C(s,M)>0$ such that, for every Banach space $E$, every
$\lambda$-embedding $I\colon E \hookrightarrow \ell_\infty^N$, and every choice of $x_1, \ldots, x_K \in E$, we have
\begin{align*}
\bigg\|\sum_{i=1}^K  \gamma_i x_i\bigg\|_{\varphi(L_\infty, L_{\varphi_2})(X)} \leq c \lambda C\,
\overline{\varphi}\big(1, \sqrt{1 + \log N}\,\big)
\sup_{1 \leq j \leq N} \big\|\big(I(x_i)(j)\big)_{i=1}^K\big\|_{\varphi(\ell_1, \ell_2)}\,.
\end{align*}
\end{Theo}

We note that if the Orlicz function $\Phi$ is defined by $\Phi(t) = e^{\varphi(1, \sqrt{t})} - 1$ for all $t\geq 0$,
then standard calculations show
\[
\varphi(L_\infty, L_{\varphi_2}) = L_\Phi\,,
\]
with universal constants of equivalence of norms. Moreover, by the well-known formula (see \cite{Ov76}), we have
\[
\varphi(\ell_1, \ell_2) = \ell_{\phi}\,,
\]
where the Orlicz function $\phi$ is given by $\phi^{-1}(t) = \varphi(t, \sqrt{t})$ for all $t\geq 0$.
\\

\section{Randomized Dirichlet polynomials}

This section is inspired by Queff\'elec's paper \cite{Queffelec}. Based on Bohr's vision of ordinary Dirichlet series and results from
the preceding sections, our goal is to provide some new $K\!S\!Z$--inequalities for  randomized Dirichlet polynomials.

Inequalities of this type recently play a~crucial role within the study of Dirichlet series -- see in particular the probabilistic proofs
of the Bohr-Bohnenblust-Hille theorem on Bohr's absolute convergence problem from   \cite[Remark 7.3]{Defant} and \cite[Theorem 5.4.2]{QQ}.
For more applications in this direction, see e.g., \cite{Defant}, \cite{Halasz}, \cite{Kahane}, and \cite{Weber}.

\medskip

Given a finite subset  $A \subset \N$, we denote by $\mathcal{D}_A$ the $|A|$-dimensional linear space of all Dirichlet polynomials $D$
defined by
\[
D(s)= \sum_{n\in A} a_n n^{-s}, \quad\, s\in \mathbb{C} \,,
\]
with complex coefficients $a_n$, $n \in A$. Since each such Dirichlet polynomial obviously defines a bounded and holomorphic function
on the right half-plane in $\mathbb{C}$, the  space $\mathcal{D}_A$ forms a~Banach space whenever it is  equipped with the norm
\[
\|D\|_\infty = \sup_{{\rm{Re}} s >0} \Big| \sum_{n=1}^N a_n n^{-s} \Big|=\sup_{t \in \R} \Big| \sum_{n=1}^N a_n n^{-it} \Big|\,.
\]
We note that the particular cases $a_n=1$ and $a_n = (-1)^n$ play a~crucial role within the study of the Riemann zeta-function
$\zeta\colon \mathbb{C} \setminus \{1\} \to \mathbb{C}$. In fact, in recent times, techniques related to random inequalities
for Dirichlet polynomials have gained more and more importance. This may be illustrated by a~deep classical result of  Tur\'an
\cite{Turan}, which states that the truth of the famous Lindel\"of's conjecture:
\[
\zeta\big(1/2 + it\big) = \mathcal{O}_{\varepsilon}(t^\varepsilon), \quad\, t\in \mathbb{R},
\]
with an arbitrarily small $\varepsilon>0$,  is equivalent to the validity of the inequality:
\[
\bigg|\sum_{n=1}^N \frac{(-1)^n}{n^{it}}\bigg| \leq C N^{\frac{1}{2} + \varepsilon}(2 + |t|)^{\varepsilon}, \quad\, t\in \mathbb{R}
\]
for an arbitrarily small $\varepsilon>0$ and with $C$ depending on $\varepsilon$.

\medskip

In order to formulate our main result we need two characteristics of the finite set $A \subset \N$ defining $\mathcal{D}_A$. For $x \geq
2$ we  denote (as usual) by $\pi(x)$ the number of all primes in the interval $[2, x]$, and by $\Omega(n)$ the number of prime divisors
of $n \in \N$ counted accorded to their multiplicities. We define
\[
\text{
$\Pi(A) := \max_{n \in A} \pi(n)$ \,\,\, and \,\,\,
$\Omega(A) := \max_{n \in A} \Omega(n)$\,. }
\]

\medskip

\begin{Theo}
\label{main}
Adopting the notation used  in Remark~$\ref{notation}$, for every $r\in [2, \infty)$ there is a~constant $C_r >0$ such
that for any finite set $A \subset \N$ and any choice of finitely many Dirichlet polynomials $D_1, \ldots, D_K \in \mathcal{D}_{A}$,
we have
\begin{equation*} \label{estimateD2}
\bigg\|\sup_{t\in \mathbb{R}} \Big|\sum_{j=1}^K  \gamma_j D_j (t)\Big| \bigg\|_{L_{\varphi_r}} \leq  C_r
\Big(1 + \Pi(A)\big(1+ 20 \log \Omega(A)\big)\Big)^{\frac{1}{r}} \sup_{t \in \mathbb{R}} \big\|(D_j(t))_{j=1}^K\big\|_{S_{r'}}\,.
\end{equation*}
\end{Theo}

Before we give the proof of this result, we state an immediate consequence  of independent interest.

\begin{Coro}
\label{DirichletKSZ}
Adopting the notation used  in Remark~$\ref{notation}$, for every $r\in [2, \infty)$ there is a~constant $C_r >0$
such that such, for every Dirichlet random polynomial $\sum_{n\in A} \gamma_n a_n n^{-it}$ in $\mathcal{D}_{A}$, we have
\begin{equation*}
\label{estimateD2}
\bigg\|\sup_{t\in \mathbb{R}} \Big|\sum_{n\in A}  \gamma_n a_n n^{-it}\Big| \bigg\|_{L_{\varphi_r}}
\leq  C_r  \Big(1 +\Pi (A)\big(1+20 \log \Omega (A)\big)\Big)^{\frac{1}{r}}  \big\|(a_n)_{n \in A}\big\|_{S_{r'}}\,.
\end{equation*}
\end{Coro}

\medskip

As mentioned, our proof of Theorem~\ref{main} is based on 'Bohr's point of view' (carefully explained in \cite{Defant} and
\cite{Queffelec, QQ}). More precisely, in our situation we need to embed  $\mathcal{D}_{A}$ into a certain  space of
trigonometric polynomials, controlling the degree as well as the number of variables of the polynomials in this space.
To achieve this, we consider the following so-called Bohr lift:
\[
\mathcal{L}_A: \mathcal{D}_{A} \to \mathcal{T}_{\Omega(A)}(\T^{\Pi(A)})\,, \,\,\,\sum_{n \in A} a_n n^{-s}
\mapsto  \sum_{ \alpha:\mathfrak{p}^\alpha \in A} a_{\mathfrak{p}^\alpha} z^\alpha\,.
\]
By (a particular case of) Kronecker's theorem on Diophantine approximation we know that the continuous homomorphism
\[
\beta: \R \to \T^{\Pi(A)}\,,\,\, t \to \big(\mathfrak{p}_k^{it}\big)_{k=1}^{\Pi(A)}
\]
has dense range (see, e.g., \cite[Proposition 3.4]{Defant} or \cite[Section 2.2]{QQ}). This implies that $\mathcal{L}_A$ is an
isometry into.

Moreover, we repeat from \eqref{bernd}  that there is a finite subset $F \subset \T^{\Pi(A)}$ with cardinality
${\rm{card}}(F) \leq N = (1 + 20\,\Omega(A))^{\Pi(A)}$ such that
\begin{align*}
\label{bernd} I\colon  \mathcal{T}_{\Omega(A)}(\T^{\Pi(A)}) \hookrightarrow \ell_\infty^N\,,\,\,\,I(P):= (P(z_i))_{i \in F},
\end{align*}
is a~$2$-isomorphic embedding. Combining all this we obtain the following embedding theorem.

\begin{Prop}
\label{crucial}
For every finite subset $A \subset \N$ there is a subset $F \subset \T^{\Pi(A)}$ with cardinality
${\rm{card}}(F) \leq N =(1 + 20 \,\Omega(A))^{\Pi(A)}$
\[
T\circ \mathcal{L}_A: \mathcal{D}_A \hookrightarrow \ell_\infty^N\, , \,\,\,
D \mapsto \Big( \big(\mathcal{L}(D)(z)\big)\Big)_{z \in F}
\]
is a $2$-embedding.
\end{Prop}

Now we easily obtain the proof of Theorem~\ref{main}.

\begin{proof}[Proof of Theorem~$\ref{main}$]
The proof is, in fact, immediate from Theorem~\ref{KSZone} (or Remark~\ref{sylvia}), taking into account that by Kronecker's
theorem we have
\begin{align*}
\sup_{z \in F}  \Big\|\Big( (I\circ\mathcal{L})\big(D_j\big)&(z)\Big)_{j=1}^K \Big\|_{S_{r'}}
\leq \sup_{z \in \T^{\Pi(A)}}  \Big\|\Big( (I\circ\mathcal{L})\big(D_j\big)(z)\Big)_{j=1}^K \Big\|_{S_{r'}} \\
&= \sup_{t \in \R}  \Big\|\Big( (I\circ\mathcal{L})\big(D_j\big)(\beta(t))\Big)_{j=1}^K \Big\|_{S_{r'}}
\leq \sup_{t \in \R}  \Big\|\Big( D_j(t)\Big)_{j=1}^K \Big\|_{S_{r'}}\,.\qedhere
\end{align*}
\end{proof}

\medskip

In the following examples we consider several interesting subclasses of all Dirichlet polynomials of length $N$,
each given by a particular finite subset $A \subset \N$:\\

\noindent {\bf Example 1.} For  $N\in \N$ and $2 \leq x \leq N$ define $$A(N,x):= \{1 \leq n \leq N;\, \pi (n) \leq x\}\,.$$
Then $\mathcal{D}_{A(N,x)}$ is the space of all Dirichlet polynomials of length $N$, which only 'depend on $\pi(x)$-many primes'.
Using the remarkably  sharp estimates for $\pi(x)$  due to Costa Periera \cite{Costa}:
\begin{align*}
\frac{x\log 2}{\log x} <\pi(x), \quad\, x\geq 5 \quad\, \text{and \,\, $\pi(x) < \frac{5x}{3\log x}$, \quad\, $x>1$}\,,
\end{align*}
we see that
\[\Pi(A(N,x)) \leq  \pi(x) < \frac{5x}{3\log x}\,.
\]
Moreover, since for any $1 \leq n=\mathfrak{p}^\alpha \leq N$ with $\alpha \in \N^{\pi(x)}$ we have that
$2^{|\alpha|} \leq N$, we get
\[
\Omega(A(N,x)) \leq \frac{\log N}{\log 2}\,.
\]
With these estimates for $\Pi(A(N,x))$ and $\Omega(A(N,x))$ our $K\!S\!Z$--inequalities from Theorem~\ref{main}
extend Queff\'elec's results from  \cite[Theorem 5.3.5]{QQ} considerably.

Let us look at the special case $x = N$, and denote by $\mathcal{D}_N$ the Banach space of all Dirichlet polynomials
of length $N$, in other words,  $\mathcal{D}_N = \mathcal{D}_{A(N)}$ with $A(N) =\{1, \ldots, N\}$. Then $\Pi(A(N))
< \frac{5N}{3\log N}$ and $\Omega(A(N)) \leq \frac{\log N}{\log 2}$. It is worth noting that in the case
$N=\mathfrak{p}_n$, the $n^{{\rm{th}}}$ prime, we have $\Pi(A_N) = n$.\\

\noindent {\bf Example 2.} Given $N$, $m \in \N$, denote by $B(N,m)$ the set of all natural numbers $1 \leq n \leq N$
which are '$m$-homogeneous' in the sense that for all $n= \mathfrak{p}^\alpha$, we have  $|\alpha|=m$ (each $n$ has
less than $m$ prime divisors, counted according to their multiplicities). Then $\mathcal{D}_{B(N,m)}$ is the space of
all $m$-homogeneous Dirichlet polynomials of length $N$. As above, we have
\[
\Pi(B(N,m)) \leq  \pi(N) < \frac{5N}{3\log N}\,,
\]
and obviously
\[
\Omega(B(N,m)) = m\,.
\]

\noindent {\bf Example 3.} A special case of the preceding result ($N = \mathfrak{p}_N$ and $m=1$)
is given by $C(N) = \{\mathfrak{p}_1, \ldots, \mathfrak{p}_N\}$\,. Then $\mathcal{D}_{C(N)}$ consists
of a~Dirichlet polynomials $\sum_{n=1}^N a_{\mathfrak{p}_n} \mathfrak{p}_n^{-s}$, and
\[
\text{$\Pi(C(N)) = N$ \,\,\,and \,\,\, $\Omega(C(N)) = 1$}\,.
\]
In passing, we note that by Bohr's inequality the linear bijection
\[
\ell_1^N \to \mathcal{D}_{C(N)},\,\,\, (a_n)_{n=1}^N \mapsto \sum_{n=1}^N a_{\mathfrak{p}_n} \mathfrak{p}_n^{-s}
\]
is isometric (\cite[Corollary 4.3]{Defant} and \cite[Theorem 4.4.1]{QQ}).

\medskip

We close the paper with following interpolation estimate for randomized Dirichlet polynomials which is a~consequence
of Proposition~\ref{crucial} and Remark~\ref{sylvia2}.

\begin{Theo}
Let $(\gamma_i)_{i \in \N}$ be a sequence of $($real or complex$)$ subgaussian random variables such that
$s = \sup_i \text{sg}(\gamma_i) < \infty$ and $M = \sup_i\|\gamma_i\|_\infty < \infty$. Suppose that an
exact interpolation functor $\mathcal{F}$ has the $\infty$-property with the constant $\delta$.

Then there is a~constant $C = C(s,M) >0$  such that, for any finite set $A \subset \N$ and any choice of
finitely many Dirichlet polynomials $D_1, \ldots, D_K \in \mathcal{D}_{A}$, we have
\begin{align*} \label{estimateD2}
\bigg\|\sup_{t\in \mathbb{R}}
\Big|\sum_{j=1}^K  &\gamma_j D_j (t)\Big| \bigg\|_{\mathcal{F}(L_\infty, L_{\varphi_2})}
\\[1ex]&
\leq  2\,\delta C \phi_{\mathcal{F}}\Big(1, \sqrt{1 + \Pi(A)\big(1+ 20 \log \Omega(A)}\Big)
\sup_{t \in \mathbb{R}} \big\|(D_j(t))_{j=1}^K\big\|_{\mathcal{F}(\ell_1, \ell_2)}\,,
\end {align*}
where $\phi_{\mathcal{F}}$ is the fundamental function of $\mathcal{F}$.
\end{Theo}

\vspace{2.5 mm}

\noindent
Institut f\"ur Mathematik \\
Carl von Ossietzky Universit\"at \\
Postfach 2503 \\
D-26111 Oldenburg, Germany

\vspace{0.5 mm}

\noindent E-mail: defant@mathematik.uni-oldenburg.de

\vspace{5.5 mm}

\noindent Faculty of Mathematics and Computer Science\\
Adam Mickiewicz University, Pozna\'n\\
Uniwersytetu Pozna{\'n}skiego 4\\
 61-614 Pozna{\'n}, Poland

\vspace{0.5 mm}

\noindent E-mail: mastylo$@$amu.edu.pl
\end{document}